\documentclass[11pt, a4paper]{article}

\usepackage{amssymb}
\usepackage{mathtools}
\usepackage[implicit=false]{hyperref}

\usepackage[backend=biber, sorting=nyt, url=false]{biblatex}
\AtEveryBibitem{\clearlist{language}}
\usepackage{tikz-cd}
\usetikzlibrary{nfold}
\tikzcdset{arrows=nfold}

\usepackage{amsthm}
\newtheorem{theorem}{Theorem}
\newtheorem{proposition}{Proposition}
\newtheorem{lemma}{Lemma}
\newtheorem{corollary}{Corollary}
\newtheorem{definition}{Definition}

\newcommand{\cat}[1]{\mathbf{#1}}
\newcommand{\bicat}[1]{\mathfrak{#1}}
\newcommand{\bifun}[1]{\mathbf{#1}}
\newcommand{\id}{\mathnormal{1}}           	
\newcommand{\Id}{1}                         		
\newcommand{\ID}{\bifun{1}}                   		

\DeclareMathOperator{\Hom}{Hom}
\newcommand{\biHom}{\cat{Hom}}

\makeatletter
\newcommand{\xRrightarrow}[2][]{\ext@arrow 0359\Rrightarrowfill@{#1}{#2}}
\newcommand{\Rrightarrowfill@}{\arrowfill@\equiv\equiv\Rrightarrow}
\makeatother

\makeatletter
\def\rightarrowfill@script{%
    \arrowfill@{\scriptstyle\relbar}{\scriptstyle\relbar}{\scriptstyle\rightarrow}}
\newcommand{\xrightrightarrows}[2][]{\mathrel{%
  \raise.54ex\hbox{%
    $\ext@arrow 3095\rightarrowfill@script{\phantom{#1}}{#2}$}%
  \setbox0=\hbox{%
    $\ext@arrow 0359\rightarrowfill@script{#1}{\phantom{#2}}$}%
  \kern-\wd0 \lower.24ex\box0}}
\makeatother

\newcommand*{\isoto}[1][]{\xrightarrow[#1]{\! \smash{\raisebox{-0.352ex}{\ensuremath{\scriptstyle\sim}}}}}
\newcommand*{\Isoto}[1][]{\xRightarrow[#1]{\smash{\raisebox{-0.352ex}{\ensuremath{\scriptstyle\sim}}}}}
\newcommand*{\Iisoto}[1][]{\xRrightarrow[#1]{\! \smash{\raisebox{-0.352ex}{\ensuremath{\scriptstyle\sim}}}}}

\newcommand{\InvCone}{\cat{InvCone}}

\DeclareMathOperator{\Kl}{\bicat{Kl}}

\newcommand{\op}{^{\mathrm{op}}}
\newcommand{\co}{^{\mathrm{co}}}

\newcommand{\Alpha}{\mathrm{A}}
\newcommand{\Epsilon}{\mathrm{E}}
\newcommand{\Eta}{\mathrm{H}}
\newcommand{\Rho}{\mathrm{P}}
\newcommand{\Tau}{\mathrm{T}}
\newcommand{\Chi}{\mathrm{X}}
\newcommand{\Kappa}{\mathrm{K}}

\title{On biadjoint triangles with additional modifications}
\author{Gabriel Merlin
\thanks{Istituto Grothendieck ETS, Corso Statuto 24, 12084 Mondovì, Italy.\\
\textit{E-mail address:} \texttt{gabriel.merlin@ctta.igrothendieck.org}}
\textsuperscript{,}%
\thanks{Université Paris-Saclay, CentraleSupélec, Mathématiques et Informatique pour la Complexité et les Systèmes, 91190, Gif-sur-Yvette, France.\\
\textit{E-mail address:} \texttt{gabriel.merlin@centralesupelec.fr}}}

\newsavebox{\tmpcontent}

\begin{document}
\maketitle

\begin{abstract}
Lucatelli Nunes obtained a 2-categorical version of the adjoint triangle theorem of Dubuc using the descent object of a specific diagram.
In some cases, such a diagram can be filled with an extra cell.
We show then how to obtain a biadjoint as an inverter of this additional datum (under suitable hypotheses).
The problem addressed here is slightly different however: we still have a triangle of pseudofunctors but the lifted biadjoint is not the same.
The construction is simplified when the pseudofunctor whose left biadjoint is sought is fully faithful.
As an example, we get the biadjoint of the inclusion pseudofunctor of a bicategory associated to a KZ-monad preserving pseudomonicity.
\end{abstract}

In category theory, several theorems provide sufficient conditions for a functor to have an adjoint.
One of these is the adjoint triangle theorem of \textcite{Dubuc_1968}.
It states that in a commuting (up to a natural isomorphism) triangle of functors
\[\begin{tikzcd}[column sep=small]
    \cat{A} \ar[rr, "F"] \ar[dr, "H" swap, ""{coordinate, name=H}] & & \cat{B} \ar[dl, "G", ""{coordinate, name=G}]
    \ar[from=G, to=H, Rightarrow, shorten=0.2em, bend right, "\sim" sloped] \\
    & \cat{C} &
\end{tikzcd},\]
if both $G$ and $H$ have a left adjoint,
if $G$ is of descent type and some coequalizers exist in $\cat{A}$,
then $F$ also has a left adjoint.

\Textcite{Nunes_2016} proved a 2-dimensional version of this theorem for biadjunction of pseudofunctors between 2-categories.
It relies on the existence of some descent objects
which replace the traditional coequalizers in this 2-dimensional setting.
As an application, he deduced a biadjunction between 2-categories of pseudo-coalgebras of some related pseudocomonads.

In some situations, the truncated simplicial object shaping the descent diagram is equipped with an extra cell.
This notably occurs when studying Kleisli bicategories associated to KZ-monads.
In this paper, we show how to use this additional datum (together with additional hypotheses)
to simplify the construction of the biadjoint.
Our main result is Theorem~\ref{thm:biadj_triangle} stating that
the sought left biadjoint can be obtained as the inverter of this additional modification.

Everything here is done for \emph{bicategories} rather than for \emph{2-categories}.
However, our result holds in a slightly modified triangle of biadjunctions
\[\begin{tikzcd}[column sep=small]
    \bicat{A} \ar[rr, "\bifun{F}"] \ar[dr, "\bifun{H}" swap, ""{coordinate, name=H}] & & \bicat{B} \ar[dl, "\bifun{G}", ""{coordinate, name=G}]
    \ar[from=G, to=H, Rightarrow, shorten=0.2em, bend right, "\sim" sloped] \\
    & \bicat{C} &
\end{tikzcd},\]
because it requires the pseudofunctor $\bifun{F}$ to have a left biadjoint instead of $\bifun{G}$.
The demonstration being mostly of a tricategorical nature,
this mostly accounts for reversing the orientation of the arrows.

In Section~\ref{sec:bicat}, we remind the bicategorical and tricategorical notions used throughout this paper.

In Section~\ref{sec:biadj_triangle}, we present the main theorem of this paper and its proof, which proceeds as follows.
Let us call $\bifun{L}$ (resp. $\bifun{N}$) the left biadjoint of $\bifun{F}$ (resp. $\bifun{H}$) and $\Phi$ the unit of $\bifun{L} \dashv \bifun{F}$.
Two pseudonatural transformations $\Phi_{\bifun{F N}}$ and $\Xi$ from $\bifun{F N}$ to $(\bifun{F L})(\bifun{F N})$ can be constructed from the initial data.
The existence of a modification $\mu$ from $\Phi_{\bifun{F N}}$ to $\Xi$ and of its inverter $(\bifun{M}, \Lambda)$ is then postulated.
A pseudonatural transformation $\ID_{\bicat{C}} \Rightarrow \bifun{G M}$ playing the role of unit
is fabricated in Subsection~\ref{subsec:unit} using the inverter property of $(\bifun{G M}, \bifun{G} \Lambda)$ which is assumed.
In Subsection~\ref{subsec:counit}, an other pseudonatural transformation $\bifun{M G} \Rightarrow \ID_{\bicat{B}}$
is designed as counit after having supposed that $(\ID_{\bicat{C}}, \Phi)$ is the inverter of an additional modification.
Finally, Subsection~\ref{subsec:triangle_laws} presents two invertible modifications translating the triangle laws of the biadjunction.
This last step requires nevertheless additional hypotheses.

Section~\ref{sec:ff_case} focuses on the special situation of a fully faithful $\bifun{G}$,
in which case the statement of our theorem becomes simpler.
As an illustration, we show that,
given a KZ-monad preserving pseudomonic 1-morphisms on a bicategory $\bicat{C}$ having all inverters,
the inclusion pseudofunctor of the full sub-bicategory of $\bicat{C}$ made of the inverters of the structural modification of the KZ-monad has a left biadjoint.

Finally, in Section~\ref{sec:dual}, we dualize our result, providing a version to lift a right biadjoint.

Many constructions developed here involve a lot of algebraic computations.
Those are easy but cumbersome, and will be frequently omitted.
Nevertheless, to ensure the correctness of this work,
all of the theorems have been implemented and checked in the formal proof assistant Coq.
The code relies on the Coq library Unimath \cite{UniMath} (notably its package Bicategories presented in \cite{Ahrens_2019}) and is available at \url{https://github.com/CentreForToposTheoryAndItsApplications/Triangle-of-biadjunctions}.

\section{Bicategorical conventions}\label{sec:bicat}

\subsection{Basics}

Conventions of \cite{Johnson_2021} will mainly be followed.
Strict notions of bicategories or pseudofunctors between them will never be examined.

A \emph{bicategory} consists of a (often large) set of objects,
categories of morphisms between objects, horizontal composition functors,
identity morphisms, invertible associators and invertible unitors satisfying the appropriate axioms.
More precisely,
\begin{definition}
A \emph{bicategory} $\bicat{B}$ consists of:
\begin{itemize}
\item a set $B_0$ whose elements are called \emph{objects};
\item a category $\biHom_{\bicat{B}}(X, Y)$ for each pair of objects $(X, Y)$;
\item for each $X \in B_0$, a particular object $\Id_X$ of $\biHom_{\bicat{B}}(X, X)$
called the \emph{identity 1-morphism} of $X$;
\item for each triple of objects $(X, Y, Z)$, a functor
\[ c_{X, Y, Z} \colon \biHom_{\bicat{B}}(Y, Z) \times \biHom_{\bicat{B}}(X, Y) \longrightarrow \biHom_{\bicat{B}}(X, Z) \]
called the \emph{horizontal composition};
\item for each quadruple of objects $(W, X, Y, Z)$, a natural isomorphism
\[\begin{tikzcd}
    \biHom_{\bicat{B}}(Y, Z) \times \biHom_{\bicat{B}}(X, Y) \times \biHom_{\bicat{B}}(W, X)
        \rar[outer sep=.4em]{\Id_{\biHom(Y, Z)} \times c_{W, X, Y}} \dar[swap]{c_{X, Y, Z} \times \Id_{\biHom(W, X)}} &
    \biHom_{\bicat{B}}(Y, Z) \times \biHom_{\bicat{B}}(W, Y) \dar{c_{W, Y, Z}} \\
    \biHom_{\bicat{B}}(X, Z) \times \biHom_{\bicat{B}}(W, X)
        \rar[swap]{c_{W, X, Z}} \ar[ur, Rightarrow, "\sim" sloped, "\alpha_{W, X, Y, Z}" swap] &
    \biHom_{\bicat{B}}(W, Z)
\end{tikzcd}\]
called the \emph{associator};
\item for each pair of objects $(X, Y)$, natural isomorphisms
\[\begin{tikzcd}
    \biHom_{\bicat{B}}(X, Y)
        \rar[outer sep=.4em]{(\Id_Y, \Id_{\biHom(X, Y)})}
        \ar[dr, bend right=10, "\Id_{\biHom(X, Y)}" swap, ""{coordinate, name=Id}] &
    \biHom_{\bicat{B}}(Y, Y) \times \biHom_{\bicat{B}}(X, Y)
        \dar{c_{X, Y, Y}} \ar[to=Id, bend right=5, Rightarrow, shorten >=0.4em, "\sim" sloped, "\lambda_{X, Y}"] \\
    & \biHom_{\bicat{B}}(X, Y)
\end{tikzcd}\]
and
\[\begin{tikzcd}
    \biHom_{\bicat{B}}(X, Y)
        \rar[outer sep=.4em]{(\Id_{\biHom(X, Y)}, \Id_X)}
        \ar[dr, bend right=10, "\Id_{\biHom(X, Y)}" swap, ""{coordinate, name=Id}] &
    \biHom_{\bicat{B}}(X, Y) \times \biHom_{\bicat{B}}(X, X)
        \dar{c_{X, X, Y}} \ar[to=Id, bend right=5, Rightarrow, shorten >=0.4em, "\sim" sloped, "\rho_{X, Y}"] \\
    & \biHom_{\bicat{B}}(X, Y)
\end{tikzcd}\]
called the \emph{left unitor} and \emph{right unitor} respectively.
\end{itemize}
These data are subject to the two following axioms:
\begin{description}
\item[Unity Axiom] given a triple of objects $(X, Y, Z)$,
for any objects $f$ and $g$ of $\biHom_{\bicat{B}}(X, Y)$ and $\biHom_{\bicat{B}}(Y, Z)$ respectively,
the diagram
\[\begin{tikzcd}[column sep=tiny]
    c_{X, Y, Z}(c_{Y, Y, Z}(g, \Id_Y), f) \ar[rr, "{\alpha_{X, Y, Y, Z}(g, \Id_Y, f)}"]
        \ar[dr, "{c_{X, Y, Z}(\rho_{Y, Z}(g), \id_f)}" swap] &&
    c_{X, Y, Z}(g, c_{X, Y, Y}(\Id_Y, f)) \ar[dl, "{c_{X, Y, Z}(\id_g, \lambda_{X, Y}(f))}"] \\
    & c_{X, Y, Z}(g, f) &
\end{tikzcd}\]
in $\biHom_{\bicat{B}}(X, Z)$ is commutative;
\item[Pentagon Axiom] given a quintuple of objects $(V, W, X, Y, Z)$,
for any objects $f$, $g$, $h$ and $i$ of
$\biHom_{\bicat{B}}(V, W)$, $\biHom_{\bicat{B}}(W, X)$, $\biHom_{\bicat{B}}(X, Y)$ and $\biHom_{\bicat{B}}(Y, Z)$ respectively,
the diagram
\[\sbox{\tmpcontent}{$c_{V, W, Z}(c_{W, X, Z}(c_{X, Y, Z}(i, h), g), f)$}
\begin{tikzcd}[row sep=large, column sep={1.2em,between origins}]
    &&[+\the\dimexpr\wd\tmpcontent/2\relax] c_{V, X, Z}(c_{X, Y, Z}(i, h), c_{V, W, X}(g, f)) \ar[drr, "{\alpha_{V, X, Y, Z}(i, h, c_{V, W, X}(g, f))}"] &[+\the\dimexpr\wd\tmpcontent/2\relax]& \\
    \usebox{\tmpcontent}
        \ar[urr, "{\alpha_{V, W, X, Z}(c_{X, Y, Z}(i, h), g, f)}"]
        \ar[dr, "{c_{V, W, Z}(\alpha_{W, X, Y, Z}(i, h, g), \id_f)}" description] &&&&
    c_{V, Y, Z}(i, c_{V, X, Y}(h, c_{V, W, X}(g, f))) \\[+1.35em]
    & c_{V, W, Z}(c_{W, Y, Z}(i, c_{W, X, Y}(h, g)), f)
        \ar[rr, "{\alpha_{V, W, Y, Z}(i, c_{W, X, Y}(h, g), f)}"{outer sep=.4em}] &&
    c_{V, Y, Z}(i, c_{V, W, Y}(c_{W, X, Y}(h, g), f)) \ar[ur, "{c_{V, Y, Z}(\id_i, \alpha_{V, W, X, Y}(h, g, f))}" description] &
\end{tikzcd}\]
in $\biHom_{\bicat{B}}(V, Z)$ is commutative.
\end{description}
\end{definition}

Throughout the paper, the objects (resp. the morphisms) of a category $\biHom_{\bicat{B}}(X, Y)$
will be called the \emph{1-morphisms} from $X$ to $Y$
(resp. the \emph{2-morphisms} from a 1-morphism to another)
and will be denoted by a single arrow as $f \colon X \to Y$
(resp. by a double arrow).
The identity morphism of an object $f$ of $\biHom_{\bicat{B}}(X, Y)$ will be written $\id_f$.
The composition of the category $\biHom_{\bicat{B}}(X, Y)$ will be called the \emph{vertical composition};
the vertical composite of two 2-morphisms
$\phi_1 \colon f_0 \Rightarrow f_1$ and $\phi_2 \colon f_1 \Rightarrow f_2$
will be denoted by $\phi_2 \circ \phi_1$.

The horizontal composition of two 1-morphisms $f \colon X \to Y$ and $g \colon Y \to Z$
will simply be denoted by $g \circ f$ instead of $c_{X, Y, Z}(g, f)$
and the horizontal composition of 2-morphisms $\phi \colon f_1 \Rightarrow f_2$
and $\psi \colon g_1 \Rightarrow g_2$ by $\psi * \phi$
which goes from $g_1 \circ f_1$ to $g_2 \circ f_2$.

When writing an associator or an unitor, the subscript specifying the objects of $\bicat{B}$ will be systematically omitted.
For example, the right unitor of a morphism $f \colon X \to Y$ will only be written $\rho_f$.
This convention will always apply thereafter
when manipulating structures incorporating functors or natural transformations
indexed by objects of a bicategory.

When composing a sequence of 1-morphisms,
a \emph{left normalized bracketing} will always be assumed.
For example, the 1-morphism defined as the composition
\[ W \xRightarrow{f} X \xRightarrow{g} Y \xRightarrow{h} Z \]
will be $h \circ (g \circ f)$ rather than $(h \circ g) \circ f$.

Pasting diagrams will be frequently used to define composite 2-morphisms of bicategories.
However, there are in general several ways to compose the 2-morphisms appearing in a pasting diagram.
Hopefully, the so-called Mac~Lane's coherence theorem for bicategories \cite[thm. 3.6.6]{Johnson_2021}
ensures that these potential composites are all equal.

Here is the bicategorical analog of the notion of functor.
\begin{definition}
A \emph{pseudofunctor} $\bifun{F}$ from a bicategory $\bicat{B}$ to another $\bicat{C}$ consists of:
\begin{itemize}
\item a map between sets of objects (also written $\bifun{F}$ by metonymy);
\item a functor $\bifun{F}_{X, Y} \colon \biHom_{\bicat{B}}(X, Y) \to \biHom_{\bicat{C}}(\bifun{F}(X), \bifun{F}(Y))$ for each pair of objects $(X, Y)$ of $\bicat{B}$;
\item a 2-isomorphism $\bifun{F}^0_X \colon \Id_{\bifun{F}(X)} \Isoto \bifun{F}(\Id_X)$ for every object $X$ of $\bicat{B}$;
\item an invertible natural transformation
\[\begin{tikzcd}[column sep=scriptsize]
    \biHom_{\bicat{B}}(Y, Z) \times \biHom_{\bicat{B}}(X, Y) \rar{c_{X, Y, Z}} \dar[swap]{\bifun{F}_{Y, Z} \times \bifun{F}_{X, Y}} &
    \biHom_{\bicat{B}}(X, Z) \dar{\bifun{F}_{X, Z}} \\
    \biHom_{\bicat{C}}(\bifun{F}(Y), \bifun{F}(Z)) \times \biHom_{\bicat{C}}(\bifun{F}(X), \bifun{F}(Y)) \rar[swap, outer sep=.4em]{c_{\bifun{F}(X), \bifun{F}(Y), \bifun{F}(Z)}}
        \ar[ur, Rightarrow, "\sim" sloped, "\bifun{F}^2_{X, Y, Z}" swap] &
    \biHom_{\bicat{C}}(\bifun{F}(X), \bifun{F}(Z))
\end{tikzcd}\]
for every triple of objects $(X, Y, Z)$ of $\bicat{B}$;
\end{itemize}
satisfying
\begin{itemize}
\item for any 1-morphism $f \colon X \to Y$ of $\bicat{B}$,
\[ \lambda_{\bifun{F}(f)} = \bifun{F}(\lambda_f) \circ \bifun{F}^2_{\Id_Y, f} \circ \left(\bifun{F}^0_Y * \id_{\bifun{F}(f)}\right),
\enskip
\rho_{\bifun{F}(f)} = \bifun{F}(\rho_f) \circ \bifun{F}^2_{f, \Id_X} \circ \left(\id_{\bifun{F}(f)} * \bifun{F}^0_X\right); \]
\item for any triple of composable 1-morphisms $f \colon W \to X$, $g \colon X \to Y$ and $h \colon Y \to Z$
of $\bicat{B}$,
\[ \bifun{F}^2_{h, g \circ f} \circ \left(\id_{\bifun{F}(h)} * \bifun{F}^2_{g, f}\right) \circ \alpha_{\bifun{F}(h), \bifun{F}(g), \bifun{F}(f)} =
\bifun{F}(\alpha_{h, g, f}) \circ \bifun{F}^2_{h \circ g, f} \circ \left(\bifun{F}^2_{h, g} * \id_{\bifun{F}(f)}\right). \]
\end{itemize}
\end{definition}

In the categorical setting, there is a notion of morphism between two functors sharing the same source and term, namely that of natural transformation.
As expected, it has an analog in the bicategorical setting.
We will depart here from the terminology of \cite{Johnson_2021}, who speaks of strong transformations,
preferring the term of \emph{pseudonatural transformation} as in \cites{Ahrens_2019}{Lack_2000}{Nunes_2016}.
\begin{definition}
A \emph{pseudonatural transformation} $\Phi$ from a pseudofunctor $\bifun{F} \colon \bicat{B} \to \bicat{C}$
to another $\bifun{G}$ consists of:
\begin{itemize}
\item a 1-morphism $\Phi_X \colon \bifun{F}(X) \to \bifun{G}(X)$ for each object $X$ of $\bicat{B}$;
\item a 2-isomorphism
\[\begin{tikzcd}
    \bifun{F}(X) \rar{\bifun{F}(f)} \dar{\Phi_X} &
    \bifun{F}(Y) \dar["\Phi_Y", ""{coordinate, near start, name=PhiY}] \\
    \bifun{G}(X) \rar["\bifun{G}(f)"{name=Gf}] &
    \bifun{G}(Y)
        \ar[from=Gf, to=PhiY, Rightarrow, bend left, shorten >=0.2em, "\sim" sloped, "\Phi_f"{near start}]
\end{tikzcd}\]
for each 1-morphism $f \colon X \to Y$ of $\bicat{B}$;
\end{itemize}
satisfying
\begin{itemize}
\item for any object $X$ of $\bicat{B}$,
\[\begin{tikzcd}[sep=large]
    \bifun{F}(X) \rar[bend left, "\bifun{F}(\Id_X)", ""{coordinate, name=FIdX}] \dar[swap]{\Phi_X} &
    \bifun{F}(X) \dar{\Phi_X} \\
    \bifun{G}(X) \rar[bend right, "\Id_{\bifun{G}(X)}"{name=IdGX}]
        \rar[bend left, "\bifun{G}(\Id_X)"{name=GIdX2}, ""{coordinate, name=GIdX1}]
        \ar[from=IdGX, to=GIdX1, Rightarrow, shorten >=0.2em, "\sim" sloped, "\bifun{G}^0_X" swap]
        \ar[from=GIdX2, to=FIdX, Rightarrow, shorten >=0.2em, "\sim" sloped, "\Phi_{\Id_X}" swap] &
    \bifun{G}(X)
\end{tikzcd} =
\begin{tikzcd}[sep=large]
    \bifun{F}(X) \rar[bend left, "\bifun{F}(\Id_X)", ""{coordinate, name=FIdX}] \dar[swap]{\Phi_X}
        \rar[bend right, "\Id_{\bifun{F}(X)}"{name=IdFX2}, ""{coordinate, name=IdFX1}]
        \ar[dr, bend right=22, "\Phi_X"{description, name=PhiX}]
        \ar[from=IdFX2, to=FIdX, Rightarrow, shorten >=0.2em, "\sim" sloped, "\bifun{F}^0_X" swap]
\ar[from=PhiX, to=IdFX1, bend right=10, Rightarrow, shorten >=0.2em, "\sim" sloped, "\rho_{\Phi_X}^{-1}"{swap, near end}] &
    \bifun{F}(X) \dar{\Phi_X} \\
    \bifun{G}(X) \rar[bend right, "\Id_{\bifun{G}(X)}"]
\ar[to=PhiX, bend right, Rightarrow, "\sim" sloped, "\lambda_{\Phi_X}"{swap, near end}] &
    \bifun{G}(X)
\end{tikzcd};\]
\item for any pair of composable 1-morphisms $f \colon X \to Y$ and $g \colon Y \to Z$ of $\bicat{B}$,
\[\begin{tikzcd}[column sep=small]
    \bifun{F}(X) \ar[rr, "{\bifun{F}(g \circ f)}", ""{coordinate, name=Fgf}] \dar[swap]{\Phi_X} &&
    \bifun{F}(Z) \dar{\Phi_Z} \\
    \bifun{G}(X) \ar[rr, "{\bifun{G}(g \circ f)}"{name=Ggf2}, ""{coordinate, name=Ggf1}]
        \ar[dr, "\bifun{G}(f)" swap]
        \ar[from=Ggf2, to=Fgf, Rightarrow, shorten >=0.2em, "\sim" sloped, "\Phi_{g \circ f}" swap] &&
    \bifun{G}(Z) \\
    & \bifun{G}(Y) \ar[ur, "\bifun{G}(g)" swap]
        \ar[to=Ggf1, Rightarrow, shorten >=0.2em, "\sim" sloped, "\bifun{G}^2_{g, f}" swap] &
\end{tikzcd} =
\begin{tikzcd}
    \bifun{F}(X) \ar[rr, "{\bifun{F}(g \circ f)}", ""{coordinate, name=Fgf}]
        \ar[dr, "\bifun{F}(f)"{description, name=Ff}] \dar[swap]{\Phi_X} &&
    \bifun{F}(Z) \dar{\Phi_Z} \\
    \bifun{G}(X) \ar[dr, "\bifun{G}(f)" swap, ""{coordinate, name=Gf}]
        \ar[from=Gf, to=Ff, Rightarrow, shorten <=0.4em, "\sim" sloped, "\Phi_f" swap] &
    \bifun{F}(Y) \ar[ur, "\bifun{F}(g)"{description, name=Fg}] \dar{\Phi_Y}
        \ar[to=Fgf, Rightarrow, shorten >=0.2em, "\sim" sloped, "\bifun{F}^2_{g, f}" swap] &
    \bifun{G}(Z) \\
    & \bifun{G}(Y) \ar[ur, "\bifun{G}(g)" swap, ""{coordinate, name=Gg}]
        \ar[from=Gg, to=Fg, Rightarrow, shorten <=0.4em, "\sim" sloped, "\Phi_g" swap] &
\end{tikzcd};\]
\item for any 2-morphism $\theta \colon f \Rightarrow g$ of $\bicat{B}$,
\[\begin{tikzcd}[sep=large]
    \bifun{F}(X) \rar[bend left, "\bifun{F}(g)", ""{coordinate, name=Fg}] \dar[swap]{\Phi_X}
        \rar[bend right, "\bifun{F}(f)"{name=Ff2}, ""{coordinate, name=Ff1}]
        \ar[from=Ff2, to=Fg, Rightarrow, shorten >=0.2em, "\sim" sloped, "\bifun{F}(\theta)" swap] &
    \bifun{F}(Y) \dar{\Phi_Y} \\
    \bifun{G}(X) \rar[bend right, "\bifun{G}(f)"{name=Gf}]
        \ar[from=Gf, to=Ff1, Rightarrow, shorten >=0.2em, "\sim" sloped, "\Phi_f" swap] &
    \bifun{G}(Y)
\end{tikzcd} =
\begin{tikzcd}[sep=large]
    \bifun{F}(X) \rar[bend left, "\bifun{F}(g)", ""{coordinate, name=Fg}] \dar[swap]{\Phi_X} &
    \bifun{F}(Y) \dar{\Phi_Y} \\
    \bifun{G}(X) \rar[bend right, "\bifun{G}(f)"{name=Gf}]
        \rar[bend left, "\bifun{G}(g)"{name=Gg2}, ""{coordinate, name=Gg1}]
        \ar[from=Gf, to=Gg1, Rightarrow, shorten >=0.2em, "\sim" sloped, "\bifun{G}(\theta)" swap]
        \ar[from=Gg2, to=Fg, Rightarrow, shorten >=0.2em, "\sim" sloped, "\Phi_g" swap] &
    \bifun{G}(Y)
\end{tikzcd}.\]
\end{itemize}
\end{definition}

The last condition of this definition ensures that, for any tuple of objects $(X, Y)$,
the data of $(\Phi_f)_{f \colon X \to Y}$ defines a natural isomorphism
\[\sbox{\tmpcontent}{$\biHom_{\bicat{B}}(X, Y)$}
\begin{tikzcd}[column sep=large]
    \usebox{\tmpcontent} \rar["\bifun{F}_{X, Y}", ""{coordinate, name=FXY}] \dar{\bifun{G}_{X, Y}} &
    \biHom_{\bicat{C}}(\bifun{F}(X), \bifun{F}(Y)) \dar{\Phi_Y \circ {-}} \\
    \makebox[\wd\tmpcontent][r]{$\biHom_{\bicat{C}}(\bifun{G}(X), \bifun{G}(Y))$} \rar["{-} \circ \Phi_X"{name=PhiX}] &
    \biHom_{\bicat{C}}(\bifun{F}(X), \bifun{G}(Y))
        \ar[from=PhiX, to=FXY, start anchor=north, Rightarrow, shorten >=0.2em, "\sim" sloped, "\Phi_{X, Y}" {swap, near start}]
\end{tikzcd}.\]

In the categorical setting, the construction of morphisms between morphisms stops at the level of natural transformations.
However, in the bicategorical setting, one can go one step further with modifications.
\begin{definition}
A \emph{modification} $\mu$ from a pseudonatural transformation $\Phi \colon \bifun{F} \Rightarrow \bifun{G}$ to another $\Psi$ consists of
a 2-morphism $\mu_X \colon \Phi_X \Rightarrow \Psi_X$ for all objects $X$ of $\bicat{B}$
such that for any 1-morphism $f \colon X \to Y$ of $\bicat{B}$,
\[\begin{tikzcd}[row sep=large, column sep=scriptsize]
    \bifun{F}(X) \rar{\bifun{F}(f)} \dar[bend right, "\Phi_X" swap, ""{coordinate, name=PhiX}] &
    \bifun{F}(Y) \dar[bend left, "\Psi_Y", ""{coordinate, name=PsiY}]
\dar[bend right, "\Phi_Y"{swap, name=PhiY1}, ""{coordinate, name=PhiY2}] \\
    \bifun{G}(X) \rar["\bifun{G}(f)"{name=Gf}] &
    \bifun{G}(Y)
        \ar[from=PhiX, to=PhiY1, Rightarrow, shorten <=0.2em, "\sim" sloped, "\Phi_f" swap]
        \ar[from=PhiY2, to=PsiY, Rightarrow, shorten =0.2em, "\mu_Y" swap]
\end{tikzcd} =
\begin{tikzcd}[row sep=large, column sep=scriptsize]
    \bifun{F}(X) \rar{\bifun{F}(f)} \dar[bend right, "\Phi_X" swap, ""{coordinate, name=PhiX}]
        \dar[bend left, "\Psi_X"{name=PsiX2}, ""{coordinate, name=PsiX1}] &
    \bifun{F}(Y) \dar[bend left, "\Psi_Y", ""{coordinate, name=PsiY}] \\
    \bifun{G}(X) \rar["\bifun{G}(f)"{name=Gf}] &
    \bifun{G}(Y)
        \ar[from=PhiX, to=PsiX1, Rightarrow, shorten =0.2em, "\mu_X" swap]
        \ar[from=PsiX2, to=PsiY, Rightarrow, shorten >=0.2em, "\sim" sloped, "\Psi_f" swap]
\end{tikzcd}.\]
\end{definition}

Given two bicategories $\bicat{A}$ and $\bicat{B}$,
pseudofunctors from $\bicat{A}$ to $\bicat{B}$, pseudonatural transformations between them
and their associate modifications form a new bicategory denoted by $\bicat{B}^{\bicat{A}}$.

The horizontal composition of two pseudonatural transformations
$\Phi \colon \bifun{F} \Rightarrow \bifun{G}$ and $\Psi \colon \bifun{G} \Rightarrow \bifun{H}$
will be denoted by $\Psi \circ \Phi$.
If $\Omega$ is a third composable pseudonatural transformation,
their overall composition is associative only up to an invertible modification
\[ \alpha_{\Omega, \Psi, \Phi} \colon (\Omega \circ \Psi) \circ \Phi \Iisoto \Omega \circ (\Psi \circ \Phi) \]
defined on an object $X$ by taking the associator $\alpha_{\Omega_X, \Psi_X, \Phi_X}$ in $\bicat{A}$.
Similarly, there are invertible modifications
\[ \lambda_{\Phi} \colon \Id_{\bifun{G}} \circ \Phi \Iisoto \Phi \quad\text{and}\quad
\rho_{\Phi} \colon \Phi \circ \Id_{\bifun{F}} \Iisoto \Phi \]
playing respectively the role of left and right unitors.

The vertical composition of two modifications
$\mu \colon \Phi \Rrightarrow \Psi$ and $\nu \colon \Psi \Rrightarrow \Omega$
will also be denoted by $\nu \circ \mu$.

\subsection{Tricategorical aspects}

The composition of pseudofunctors defines a pseudofunctor from the product bicategory $\bicat{B}^{\bicat{A}} \times \bicat{C}^{\bicat{B}}$ to $\bicat{C}^{\bicat{A}}$.
The composite of two pseudofunctors $\bifun{F} \colon \bicat{A} \to \bicat{B}$ and
$\bifun{G} \colon \bicat{B} \to \bicat{C}$ will be denoted without composition symbol as $\bifun{G F}$.

\paragraph{Whiskering}
Fixing a pseudofunctor $\bifun{F}$ object of $\bicat{B}^{\bicat{A}}$,
one obtains a pseudofunctor $\bicat{C}^{\bicat{B}} \to \bicat{C}^{\bicat{A}}$
whose local functors encode the \emph{pre-whiskering} by $\bifun{F}$.
The \emph{pre-whiskering} of a pseudonatural transformation
\[ \begin{tikzcd}[column sep=large]
    \bicat{B}
        \rar[bend left, start anchor=north east, end anchor=north west, "\bifun{G}_1", ""{coordinate, name=G1}]
        \rar[bend right, start anchor=south east, end anchor=south west, "\bifun{G}_2" below, ""{coordinate, name=G2}]
        \ar[from=G1, to=G2, Rightarrow, shorten=0.2em, "\Psi"] &
    \bicat{C}
\end{tikzcd} \]
by $\bifun{F}$ will be denoted $\Psi_{\bifun{F}}$.
Its component at an object $X$ of $\bicat{A}$ is
\[ \Psi_{\bifun{F}(X)} \colon \bifun{G}_1 \bifun{F}(X) \longrightarrow \bifun{G}_2 \bifun{F}(X) \]
whereas its component at a morphism $f \colon X \to Y$ is
\[\Psi_{\bifun{F}(f)} \colon \bifun{G}_2 \bifun{F}(f) \circ \Psi_{\bifun{F}(X)} \Isoto \Psi_{\bifun{F}(Y)} \circ \bifun{G}_1 \bifun{F}(f),\]
hence the index notation.
By functoriality, any modification $\nu \colon \Psi_1 \Rightarrow \Psi_2$ has a pre-whiskering by $\bifun{F}$ too,
denoted by $\nu_{\bifun{F}}$ and defined on an object $X$ as $\nu_{\bifun{F}(X)}$.

Consider now an object $\bifun{G}$ of $\bicat{C}^{\bicat{B}}$.
The \emph{post-whiskering} of a pseudonatural transformation
\[ \begin{tikzcd}[column sep=large]
    \bicat{A}
        \rar[bend left, start anchor=north east, end anchor=north west, "\bifun{F}_1", ""{coordinate, name=F1}]
        \rar[bend right, start anchor=south east, end anchor=south west, "\bifun{F}_2" below, ""{coordinate, name=F2}]
        \ar[from=F1, to=F2, Rightarrow, shorten=0.2em, "\Phi"] &
    \bicat{B}
\end{tikzcd} \]
by $\bifun{G}$ will be denoted by $\bifun{G} \Phi$.
Its component at an object $X$ of $\bicat{A}$ is
\[ \bifun{G}(\Phi_X) \colon \bifun{G} \bifun{F}_1(X) \longrightarrow \bifun{G} \bifun{F}_2(X). \]
Its component at a morphism $f \colon X \to Y$ has for expression
\[ \bifun{G} \bifun{F}_2(f) \circ \bifun{G}(\Phi_{X}) \Isoto
\bifun{G} \left( \bifun{F}_2(f) \circ \Phi_{X} \right) \Isoto[\bifun{G} (\Phi_f)]
\bifun{G} \left( \Phi_{Y} \circ \bifun{F}_1(f) \right) \Isoto
\bifun{G} (\Phi_{Y}) \circ \bifun{G} \bifun{F}_1(f). \]
The post-whiskering of a modification $\mu \colon \Phi_1 \Rightarrow \Phi_2$ is given on $X$ by
\[ (\bifun{G} \mu)_X = \bifun{G} (\mu_X). \]

Consider finally two horizontally composable pseudonatural transformations
\[ \begin{tikzcd}[column sep=large]
    \bicat{A}
        \rar[bend left, start anchor=north east, end anchor=north west, "\bifun{F}_1", ""{coordinate, name=F1}]
        \rar[bend right, start anchor=south east, end anchor=south west, "\bifun{F}_2" below, ""{coordinate, name=F2}]
        \ar[from=F1, to=F2, Rightarrow, shorten=0.2em, "\Phi"] &
    \bicat{B}
        \rar[bend left, start anchor=north east, end anchor=north west, "\bifun{G}_1", ""{coordinate, name=G1}]
        \rar[bend right, start anchor=south east, end anchor=south west, "\bifun{G}_2" below, ""{coordinate, name=G2}]
        \ar[from=G1, to=G2, Rightarrow, shorten=0.2em, "\Psi"] &
    \bicat{C}
\end{tikzcd}. \]
The exchange property of post-whiskering and pre-whiskering is replaced by an invertible modification
\[ \Psi_{\Phi} \colon \bifun{G}_2(\Phi) \circ \Psi_{\bifun{F}_1} \Iisoto \Psi_{\bifun{F}_2} \circ \bifun{G}_1(\Phi) \]
defined at any object $X$ of $\bicat{A}$ by
\[\begin{tikzcd}
    \bifun{G}_1 \bifun{F}_1(X) \rar["\bifun{G}_1(\Phi_X)", ""{coordinate, name=G1}] \dar[swap]{\Psi_{\bifun{F}_1(X)}} &
    \bifun{G}_1 \bifun{F}_2(X) \dar{\Psi_{\bifun{F}_2(X)}} \\
    \bifun{G}_2 \bifun{F}_1(X) \rar["\bifun{G}_2 (\Phi_X)"{name=G2}] &
    \bifun{G}_2 \bifun{F}_2(X)
    \ar[from=G2, to=G1, Rightarrow, shorten >=0.2em, "\sim" sloped, "\Psi_{\Phi_X}" swap]
\end{tikzcd}. \]

\paragraph{Associators}
The composition of three pseudofunctors $\bifun{F}$, $\bifun{G}$, $\bifun{H}$ is associative only up to a pseudonatural transformation
\[ \Alpha_{\bifun{H}, \bifun{G}, \bifun{F}} \colon (\bifun{H G}) \bifun{F} \Iisoto \bifun{H} (\bifun{G F}) \]
which is an equivalence as 1-morphism in the suitable bicategory of pseudofunctors.
It has a canonical quasi-inverse simply denoted by $\Alpha_{\bifun{H}, \bifun{G}, \bifun{F}}^{-1}$.

The dependency in $\bifun{F}$, $\bifun{G}$ and $\bifun{H}$ of this associator is pseudonatural,
in the sense that it defines a pseudonatural transformation filling
\[\begin{tikzcd}
    \bicat{D}^{\bicat{C}} \times \bicat{C}^{\bicat{B}} \times \bicat{B}^{\bicat{A}} \rar \dar &
    \bicat{D}^{\bicat{C}} \times \bicat{C}^{\bicat{A}} \dar \\
    \bicat{D}^{\bicat{B}} \times \bicat{B}^{\bicat{A}} \rar
        \ar[ur, Rightarrow, "\Alpha"] &
    \bicat{D}^{\bicat{A}}
\end{tikzcd}.\]
As a consequence, if there is for example a pseudonatural transformation $\Phi \colon \bifun{F}_1 \Rightarrow \bifun{F}_2$,
it induces an invertible modification $\Alpha_{\bifun{H}, \bifun{G}, \Phi}$ filling
\[\begin{tikzcd}[arrows=Rightarrow]
    (\bifun{H G}) \bifun{F}_1 \rar{(\bifun{H G}) \Phi} \dar[swap]{\Alpha_{\bifun{H}, \bifun{G}, \bifun{F}_1}} &
    (\bifun{H G}) \bifun{F}_2 \dar{\Alpha_{\bifun{H}, \bifun{G}, \bifun{F}_2}} \\
    \bifun{H} (\bifun{G} \bifun{F}_1) \rar[swap]{\bifun{H} (\bifun{G} \Phi)} \ar[ur, nfold=3, "\sim" sloped, "\Alpha_{\bifun{H}, \bifun{G}, \Phi}" swap] &
    \bifun{H} (\bifun{G} \bifun{F}_2)
\end{tikzcd}.\]

If $\bifun{F}$, $\bifun{G}$, $\bifun{H}$, $\bifun{I}$ are four composable pseudofunctors,
the well-known pentagon identity is replaced by a \emph{pentagonator},
which is an invertible modification
\[\sbox{\tmpcontent}{$(\bifun{I H})(\bifun{G F})$}
\begin{tikzcd}[arrows=Rightarrow, column sep={\the\wd\tmpcontent,between origins}, row sep=large]
	&[-3em] {(\bifun{I} (\bifun{H G})) \bifun{F}} && {\bifun{I} ((\bifun{H G}) \bifun{F})} &[-3em] \\
	{((\bifun{I H}) \bifun{G}) \bifun{F}} & & & & {\bifun{I} (\bifun{H} (\bifun{G F}))} \\
	& & \usebox{\tmpcontent}
	\ar[from=2-1, to=1-2, "\Alpha_{\bifun{I}, \bifun{H}, \bifun{G}} * \bifun{F}"]
	\ar[from=1-2, to=1-4, "\Alpha_{\bifun{I}, \bifun{H G}, \bifun{F}}", ""{name=0, coordinate}]
	\ar[from=1-4, to=2-5, "\bifun{I} * \Alpha_{\bifun{I H}, \bifun{G}, \bifun{F}}"]
	\ar[from=2-1, to=3-3, "\Alpha_{\bifun{I H}, \bifun{G}, \bifun{F}}" swap]
	\ar[from=3-3, to=2-5, "\Alpha_{\bifun{I}, \bifun{H}, \bifun{G F}}" swap]
	\ar[shorten <=0.2em, nfold=3, from=0, to=3-3, "\sim" sloped]
\end{tikzcd}.\]

\paragraph{Unitors}
Let us restrict to two bicategories $\bicat{A}$ and $\bicat{B}$.
The \emph{left unitor} is a pseudonatural transformation $\Lambda$
from the pseudofunctor of post-composition by $\ID_{\bicat{B}}$
to the identity pseudofunctor of $\bicat{B}^{\bicat{A}}$.

Similarly, the \emph{right unitor} is a pseudonatural transformation $\Rho$
from the pseudofunctor of pre-composition by $\ID_{\bicat{A}}$
to the identity pseudofunctor of $\bicat{B}^{\bicat{A}}$.

These pseudonatural transformations both have quasi-inverse,
respectively denoted by $\Lambda^{-1}$ and $\Rho^{-1}$.

Let us finish this subsection by introducing some useful abuses of notation.
Thereafter, we will frequently encounter pre-whiskering of a pseudonatural transformation
$\Psi \colon \bifun{G} \Rightarrow \ID_{\bicat{B}}$ by some pseudofunctor $\bifun{F}$.
This results in a first pseudonatural transformation
$\bifun{G F} \Rightarrow \ID_{\bicat{B}} \bifun{F}$,
which can then be post-composed by $\Lambda_{\bifun{F}}^{-1}$ 
to get a pseudonatural transformation
$\bifun{G F} \Rightarrow \bifun{F}$.
In such a setting, $\Psi_{\bifun{F}}$ will always designate that last pseudonatural transformation instead of the usual one.

Similarly, if $\Phi$ is a pseudonatural transformation $\ID_{\bicat{A}} \Rightarrow \bifun{F}$
and $\bifun{G}$ a pseudofunctor starting from $\bicat{A}$,
$\bifun{G}(\Phi)$ will refer to the composite
\[ \bifun{G} \Isoto[\Rho_{\bifun{G}}^{-1}] \bifun{G} \ID_{\bicat{A}} \Rightarrow \bifun{G F}. \]

These abuses of notation enjoy the same exchange properties as the usual one.
For example, if we have
\[ \begin{tikzcd}[column sep=large]
    \bicat{A}
        \rar[bend left, start anchor=north east, end anchor=north west, "\bifun{F}_1", ""{coordinate, name=F1}]
        \rar[bend right, start anchor=south east, end anchor=south west, "\bifun{F}_2" below, ""{coordinate, name=F2}]
        \ar[from=F1, to=F2, Rightarrow, shorten=0.2em, "\Phi"] &
    \bicat{B}
        \rar[bend left, start anchor=north east, end anchor=north west, "\bifun{G}", ""{coordinate, name=G1}]
        \rar[bend right, start anchor=south east, end anchor=south west, "\ID_{\bicat{B}}" below, ""{coordinate, name=G2}]
        \ar[from=G1, to=G2, Rightarrow, shorten=0.2em, "\Psi"] &
    \bicat{B}
\end{tikzcd}, \]
there is still an invertible modification
\[ \Phi \circ \Psi_{\bifun{F}_1} \Iisoto \Psi_{\bifun{F}_2} \circ \bifun{G}(\Phi). \]

\paragraph{Biadjunctions and Pseudomonads}
These notations allow to define biadjunctions in a compact manner.
\begin{definition}
A \emph{biadjunction} between bicategories $\bicat{A}$ and $\bicat{B}$ consists in the data of:
\begin{itemize}
\item two pseudofunctors $\bifun{L} \colon \bicat{A} \to \bicat{B}$ and $\bifun{R} \colon \bicat{B} \to \bicat{A}$;
\item two pseudonatural transformations $\Eta \colon \ID_{\bicat{B}} \Rightarrow \bifun{R L}$
and $\Epsilon \colon \bifun{L R} \Rightarrow \ID_{\bicat{A}}$;
\item two invertible modifications
\begin{gather}
\begin{tikzcd}[arrows=Rightarrow, ampersand replacement=\&]
    \& \bifun{L} (\bifun{R L}) \rar["\sim" sloped, "\Alpha_{\bifun{L}, \bifun{R}, \bifun{L}}^{-1}"{swap, name=Assoc}] \&
    (\bifun{L R}) \bifun{L} \ar[dr, "\Epsilon_{\bifun{L}}"] \& \\
    \bifun{L} \ar[ur, "\bifun{L}(\Eta)"] \ar[rrr, "\Id_{\bifun{L}}"{name=IdR}] \& \& \& \bifun{L}
    \ar[from=Assoc, to=IdR, nfold=3, dashed, "\sim" sloped]
\end{tikzcd}
\intertext{and}
\begin{tikzcd}[arrows=Rightarrow, ampersand replacement=\&]
    \& (\bifun{R L}) \bifun{R} \rar["\sim" sloped, "\Alpha_{\bifun{R}, \bifun{L}, \bifun{R}}"{swap, name=Assoc}] \&
    \bifun{R} (\bifun{L R}) \ar[dr, "\bifun{R}(\Epsilon)"] \& \\
    \bifun{R} \ar[ur, "\Eta_{\bifun{R}}"] \ar[rrr, "\Id_{\bifun{R}}"{name=IdL}] \& \& \& \bifun{R}
    \ar[from=Assoc, to=IdL, nfold=3, dashed, "\sim" sloped]
\end{tikzcd}.
\end{gather}
\end{itemize}
\end{definition}
Such a biadjunction induces for any $A$ in $\bicat{A}$ and $B$ in $\bicat{B}$ functors
\begin{gather*}
    \biHom_{\bicat{B}}(\bifun{L}(A), B) \xrightarrow{\bifun{R}({-}) \circ \Eta_A} \biHom_{\bicat{A}}(A, \bifun{R}(B)),\\
    \biHom_{\bicat{A}}(A, \bifun{R}(B)) \xrightarrow{\Epsilon_B \circ \bifun{L}({-})} \biHom_{\bicat{B}}(\bifun{L}(A), B)
\end{gather*}
quasi-inverse one of each other.

This definition gives an incoherent notion of biadjunction
(called \emph{i-weak i-quasi-adjunction} in \cite{Gray_1974}).
A coherent version of biadjunction (\emph{strict i-weak i-quasi-adjunction} in the terminology of \cite{Gray_1974}) is obtained by requiring the triangulators to satisfy the so-called swallowtail equations.
They ensure that equivalences between hom categories are adjoint equivalences.
In fact, any incoherent biadjunction can be upgraded into a coherent one
by changing one of the two invertible modifications \cite{Gray_1974, Zoberlein_1976}.

A biadjunction creates a \emph{pseudomonad} \cite{Lack_2000}.
Recall that a pseudomonad on a bicategory $\bicat{A}$ consists in
a pseudofunctor $\bifun{T} \colon \bicat{A} \to \bicat{A}$,
two pseudonatural transformations $\Pi \colon \bifun{T}^2 \Rightarrow \bifun{T}$,
$\Psi \colon \ID_{\bicat{A}} \Rightarrow \bifun{T}$,
and three invertible modifications $\beta \colon \Pi \circ \Psi_{\bifun{T}} \Iisoto \Id_{\bifun{T}}$,
$\eta \colon \Pi \circ \bifun{T}(\Psi) \Iisoto \Id_{\bifun{T}}$,
$\pi \colon \Pi \circ \bifun{T}(\Pi) \circ \Alpha_{\bifun{T}, \bifun{T}, \bifun{T}} \Iisoto \Pi \circ \Pi_{\bifun{T}}$,
satisfying two coherence conditions.

\subsection{Inverters}

If $\phi$ is a 2-morphism between two parallel 1-morphisms $f_1, f_2 \colon X_1 \to X_2$ in a bicategory $\bicat{C}$,
an \emph{inverter cone} is an object $X_0$ of $\bicat{C}$ (called the \emph{vertex} of the cone)
together with a 1-morphism $f_0 \colon X_0 \to X_1$ such that $\phi * f_0$ becomes invertible.
Given an object $X_0$, inverter cones of vertex $X_0$ constitute
a full subcategory $\InvCone(X_0; \phi)$ of $\biHom_{\bicat{C}}(X_0, X_1)$,
which depends pseudofunctorially on $X_0$.
An \emph{inverter} of $\phi$ is an universal inverter cone,
that is an inverter cone $(\tilde{X}_0, \tilde{f}_0)$ such that, for any other object $X_0$,
the pre-composition functor by $\tilde{f}_0$ induces an equivalence
\[ \cat{Hom}_{\bicat{C}}(X_0, \tilde{X}_0) \isoto \InvCone(X_0; \phi). \]
It immediately implies that $\tilde{f}_0$ is a \emph{fully faithful} 1-morphism of $\bicat{C}$,
i.e. for any object $Y$ of $\bicat{C}$, the functor of precomposition by $\tilde{f}_0$
\[ \biHom_{\bicat{C}}(X_1, Y) \xrightarrow{{-} \circ \tilde{f}_0} \biHom_{\bicat{C}}(\tilde{X}_0, Y) \]
is fully faithful in the usual meaning.

Inverters are a kind of weighted bilimits \cite{Street_1976, Street_1980}.
The following proposition expresses the pseudofunctoriality of inverters with respect to $\phi$.
\begin{proposition}\label{prop:inverter_morphism}
Let $\bicat{C}$ be any bicategory in which there is a diagram
\begin{equation}
\begin{tikzcd}[row sep=large]
    X_0 \rar{f_0}
    &
    X_1
        \rar[bend left, start anchor=north east, end anchor=north west, "f_1"{description, near start, name=f1}]
        \dar["h_1"{swap, name=h1}] &[+1em]
    X_2 \dar{h_2} \\
    Y_0 \rar["g_0" name=g0]
    &
    Y_1
        \rar[bend left, start anchor=north east, end anchor=north west, "g_1"{description, near start, name=g1}]
        \ar[from=f1, to=g1, Rightarrow, shorten=0.2em, "\sim"{near start, sloped}, "\eta_1"{near end, swap}]
        \ar[from=u, to=ur, crossing over, bend right, start anchor=south east, end anchor=south west, "f_2"{description, near end, name=f2}]
        \ar[from=f1, to=f2, Rightarrow, shorten=0.2em, "\phi"]
        \rar[bend right, start anchor=south east, end anchor=south west, "g_2"{description, near end, name=g2}]
        \ar[from=g1, to=g2, Rightarrow, shorten=0.2em, "\psi" swap] &
    Y_2
        \ar[from=f2, to=g2, crossing over, Rightarrow, shorten=0.2em, "\sim"{swap, sloped}, "\eta_2"{near start}]
\end{tikzcd}
\end{equation}
with
$\eta_2 \circ (h_2 * \phi) = (\psi * h_1) \circ \eta_1$.
\begin{enumerate}
\item If $\phi * f_0$ is invertible, then $\psi * (h_1 \circ f_0)$ is also.
\item If moreover $(Y_0, g_0)$ is an inverter of $\psi$,
then there exists a 1-morphism $h_0 \colon X_0 \to Y_0$ and
a 2-isomorphism $h_1 \circ f_0 \Isoto g_0 \circ h_0$,
which are unique up to a compatible 2-isomorphism.
\end{enumerate}
\end{proposition}
\begin{proof}
The first assertion follows from the rewriting of $\psi * (h_1 \circ f_0)$ as
\begin{align*}
    \psi * (h_1 \circ f_0) &= \alpha_{g_2, h_1, f_0} \circ \left((\psi * h_1) * f_0 \right) \circ \alpha_{g_1, h_1, f_0}^{-1} \\
    &= \alpha_{g_2, h_1, f_0} \circ \left(\left(\eta_2 \circ (h_2 * \phi) \circ \eta_1^{-1}\right) * f_0 \right)
        \circ \alpha_{g_1, h_1, f_0}^{-1} \\
    &= \begin{multlined}[t]
    \alpha_{g_2, h_1, f_0} \circ (\eta_2 * f_0) \circ \alpha_{h_2, f_2, f_0}^{-1} \circ \left(h_2 * (\phi * f_0)\right) \\
    \circ \alpha_{h_2, f_1, f_0} \circ (\eta_1^{-1} * f_0) \circ \alpha_{g_1, h_1, f_0}^{-1}.
    \end{multlined}
\end{align*}
If $\phi * f_0$ is invertible, $(X_0, h_1 \circ f_0)$ is therefore an inverter cone of $\phi$.
The second assertion is now a direct consequence of the universal property of inverters.
\end{proof}

In the rest of this section, we are going to describe how inverters of modifications can be obtained from pointwise inverters.
Beforehand, we need the following lemma.

\begin{lemma}
Let $\Phi$ be a pseudonatural transformation between two pseudofunctors $\bifun{F}_1, \bifun{F}_2 \colon \bicat{A} \to \bicat{B}$.
If for any object $X$ of $\bicat{A}$, $\Phi_X$ is a fully faithful 1-morphism of $\bicat{B}$,
then $\Phi$ is fully faithful within $\bicat{B}^{\bicat{A}}$.
\end{lemma}
\begin{proof}
One has to show that, for any pseudofunctor $\bifun{G}$, the post-composition functor by $\Phi$
\[ \biHom_{\bicat{B}^{\bicat{A}}}(\bifun{G}, \bifun{F}_1) \to \biHom_{\bicat{B}^{\bicat{A}}}(\bifun{G}, \bifun{F}_2) \]
is fully faithful.

Let $\Psi_1$ and $\Psi_2$ be two pseudonatural transformations $\bifun{G} \Rightarrow \bifun{F}_1$.
If $\mu$ is a modification $\Phi \circ \Psi_1 \Rrightarrow \Phi \circ \Psi_2$,
then for any object $X$ of $\bicat{A}$,
there is an unique 2-morphism $\tilde{\mu}_X \colon (\Psi_1)_X \Rightarrow (\Psi_2)_X$
such that $\Phi_X * \tilde{\mu}_X = \mu_X$.
Since for any 1-morphism $f \colon X \to Y$, a straightforward computation shows that
\[
    \Phi_Y * \left( (\tilde{\mu}_Y * \bifun{F}_1(f)) \circ (\Psi_1)_f \right)
    = \Phi_Y * \left( (\Psi_2)_f \circ (\bifun{F}_2(f) * \tilde{\mu}_X) \right),
\]
$\tilde{\mu}$ constitutes actually a modification.

Moreover, if $\nu$ is an other modification $\Psi_1 \Rrightarrow \Psi_2$ such that $\Phi * \nu = \mu$,
then we have $\nu_X = \tilde{\mu}_X$ for any object $X$ of $\bicat{A}$, hence $\nu = \tilde{\mu}$.
\end{proof}

\begin{proposition}
Let $\bicat{A}$ and $\bicat{B}$ be two bicategories
and a cell in $\bicat{B}^{\bicat{A}}$
\[\begin{tikzcd}[arrows=Rightarrow, column sep=large]
    \bifun{F}_1
        \rar[bend left, start anchor=north east, end anchor=north west, "\Phi_1", ""{coordinate, name=Phi1}]
        \rar[bend right, start anchor=south east, end anchor=south west, "\Phi_2" below, ""{coordinate, name=Phi2}]
        \ar[from=Phi1, to=Phi2, nfold=3, shorten=0.2em, "\mu"] &
    \bifun{F}_2
\end{tikzcd}.\]
If $\Phi_0$ is a pseudonatural transformation from a pseudofunctor $\bifun{F}_0$ to $\bifun{F}_1$
which is a \emph{pointwise inverter} of $\mu$,
i.e. for any object $X$ of $\bicat{A}$, $(\bifun{F}_0(X), (\Phi_0)_X)$ is an inverter of $\mu_X$ in $\bicat{B}$,
then $(\bifun{F}_0, \Phi_0)$ is an inverter of $\mu$ in $\bicat{B}^{\bicat{A}}$.
\end{proposition}
\begin{proof}
Let $\bifun{G}$ be any object of $\bicat{B}^{\bicat{A}}$.
One has to build a functor
\[ \InvCone(\bifun{G}; \mu) \to \biHom_{\bicat{B}^{\bicat{A}}}(\bifun{G}, \bifun{F}_0) \]
quasi-inverse to the functor of pre-composition by $\Phi_0$.

Let $(\bifun{G}, \Psi)$ be an inverter cone of $\mu$.
The universal property of $(\bifun{F}_0(X), (\Phi_0)_X)$ gives, for any object $X$ of $\bicat{A}$,
a dashed 1-morphism and a dashed 2-isomorphism filling
\[\begin{tikzcd}
    \bifun{G}(X) \ar[dr, "\Psi_X", ""{coordinate, name=Psi}] \dar[dashed, "\widetilde{\Psi}_X"] & \\
    \bifun{F}_0(X) \rar[""{coordinate, name=Phi}, "(\Phi_0)_X" swap] & \bifun{F}_1(X)
    \ar[from=Psi, to=Phi, Rightarrow, bend right, dashed, shorten=0.25em, "\sim" sloped, "\eta_X"{near start}]
\end{tikzcd}.\]
For any 1-morphism $f \colon X \to Y$, the full faithfulness of $(\Phi_0)_Y$
ensures the existence of an unique 2-isomorphism
\[ \widetilde{\Psi}_f \colon \bifun{F}_0(f) \circ \widetilde{\Psi}_X \Isoto \widetilde{\Psi}_Y \circ \bifun{G}(f) \]
such that $(\Phi_0)_Y * \widetilde{\Psi}_f$ equals the composite of
\[\begin{tikzcd}
    & \bifun{F}_0(X) \ar[r, "\bifun{F}_0(f)", ""{coordinate, name=F0f}] \dar["(\Phi_0)_{\mathrlap{X}}" description] &
    \bifun{F}_0(Y) \dar["(\Phi_0)_Y"] \\
    \bifun{G}(X) \ar[ur, "\widetilde{\Psi}_X", ""{coordinate, near end, name=SigmaX}] \rar[swap]{\Psi_X} \ar[dr, "\bifun{G}(f)" swap] &
    \bifun{F}_1(X) \ar[r, "\bifun{F}_1(f)"{name=F1f}]
        \ar[from=SigmaX, Rightarrow, shorten <=0.2em, "\sim" sloped, "\eta_X^{-1}"{swap, near start}]
        \ar[from=F0f, to=F1f, Rightarrow, shorten <=0.2em, "\sim" sloped, "\Phi_f^{-1}"{swap}] &
    \bifun{F}_1(Y) \\
    & \bifun{G}(Y) \ar[r, "\widetilde{\Psi}_Y" swap] \ar[ur, "\Psi_Y", ""{coordinate, near end, name=PsiY}]
        \ar[from=u, Rightarrow, "\sim" sloped, "\Psi_f" swap] &
    \bifun{F}_0(Y) \ar[u, "(\Phi_0)_Y" swap]
        \ar[from=PsiY, Rightarrow, shorten <=0.2em, "\sim" sloped, "\eta_Y" swap]
\end{tikzcd}. \]
One can show that these data constitute
a pseudonatural transformation $\widetilde{\Psi} \colon \bifun{G} \Rightarrow \bifun{F}_0$ and
an invertible modification $\eta_{\Psi} \colon \Psi \Iisoto \Phi_0 \circ \widetilde{\Psi}$.

Consider now a modification $\nu$ between two inverter cones $(\bifun{G}, \Psi_1)$ and $(\bifun{G}, \Psi_2)$.
By full faithfulness of $\Phi_0$ within $\bicat{B}^{\bicat{A}}$, there is
an unique modification $\tilde{\nu} \colon \widetilde{\Psi}_1 \Rrightarrow \widetilde{\Psi}_2$
such that $\Phi_0 * \tilde{\nu} = \eta_{\Psi_2}^{-1} \circ \nu \circ \eta_{\Psi_1}$.

It is easy to see that these two maps $\Psi \mapsto \widetilde{\Psi}$
and $\nu \mapsto \widetilde{\nu}$ build up a functor
from $\InvCone(\bifun{G}; \mu)$ to $\biHom_{\bicat{B}^{\bicat{A}}}(\bifun{G}, \bifun{F}_0)$.
In addition, $\eta$ defines by construction a natural isomorphism.
It only remains to build invertible modifications
$\epsilon_{\Omega} \colon \widetilde{\Phi_0 * \Omega} \Rightarrow \Omega$
natural in $\Omega \colon \bifun{G} \Rightarrow \bifun{F}_0$,
but it is again a consequence of the full faithfulness of $\Phi_0$.
\end{proof}

\begin{corollary}
Let $\bicat{A}$ and $\bicat{B}$ be two bicategories
and a cell in $\bicat{B}^{\bicat{A}}$
\[\begin{tikzcd}[arrows=Rightarrow, column sep=large]
    \bifun{F}_1
        \rar[bend left, start anchor=north east, end anchor=north west, "\Phi_1", ""{coordinate, name=Phi1}]
        \rar[bend right, start anchor=south east, end anchor=south west, "\Phi_2" below, ""{coordinate, name=Phi2}]
        \ar[from=Phi1, to=Phi2, nfold=3, shorten=0.2em, "\mu"] &
    \bifun{F}_2
\end{tikzcd}.\]
If for any object $X$ of $\bicat{A}$, $\mu_X$ has an inverter $(F_0(X), \Phi_0(X))$ in $\bicat{B}$,
then there exists a pseudofunctor $\bifun{F}_0$ and
a pseudonatural transformation $\Phi_0 \colon \bifun{F}_0 \Rightarrow \bifun{F}_1$
coinciding at any $X$ with $F_0(X)$ and $\Phi_0(X)$ such that
$(\bifun{F}_0, \Phi_0)$ is an inverter of $\mu$ in $\bicat{B}^{\bicat{A}}$.
\end{corollary}
\begin{proof}
Thanks to the previous proposition,
it suffices to build the pseudofunctor $\bifun{F}_0$ and the pseudonatural transformation $\Phi_0$.
If $f$ is a 1-morphism $X \to Y$ in $\bicat{A}$, Proposition~\ref{prop:inverter_morphism} yields
a 1-morphism $\bifun{F}_0(f) \colon \bifun{F}_0(X) \to \bifun{F}_0(Y)$
and a 2-isomorphism $\Phi_0(f) \colon \bifun{F}_1(f) \circ \Phi_0(X) \Isoto \Phi_0(Y) \circ \bifun{F}_0(f)$.

If $g$ is an other 1-morphism $Y \to Z$, one needs a 2-isomorphism
\[ \bifun{F}_0^2(g, f) \colon \bifun{F}_0(g) \circ \bifun{F}_0(f) \Isoto \bifun{F}_0(g \circ f). \]
By full faithfulness of $\Phi_0(Z)$, it is equivalent to build a 2-isomorphism
\[ \Phi_0(Z) \circ \left(\bifun{F}_0(g) \circ \bifun{F}_0(f)\right) \Isoto \Phi_0(Z) \circ \bifun{F}_0(g \circ f). \]
To do so, just take
\[\begin{tikzcd}[arrows={Rightarrow, "\sim" sloped}, column sep=-2.4em]
    \Phi_0(Z) \circ \left(\bifun{F}_0(g) \circ \bifun{F}_0(f)\right) \ar[dr] & &
    \bifun{F}_1(g \circ f) \circ \Phi_0(X) \ar[dr, "\Phi_0(g \circ f)"{near end}] & \\
    & \left(\bifun{F}_1(g) \circ \bifun{F}_1(f)\right) \circ \Phi_0(X) \ar[ur, "{\bifun{F}_1^2(g, f) * \Phi_0(X)}" swap] & &
    \Phi_0(Z) \circ \bifun{F}_0(g \circ f) 
\end{tikzcd}\]
where the leftmost 2-isomorphism is
\[\begin{tikzcd}[arrows={Rightarrow, "\sim" sloped}, row sep=large, column sep=-4.8em]
    \Phi_0(Z) \circ \left(\bifun{F}_0(g) \circ \bifun{F}_0(f)\right) \ar[dr] & &
    \left( \bifun{F}_1(g) \circ \Phi_0(Y) \right) \circ \bifun{F}_0(f) \ar[dr] & &
    \left(\bifun{F}_1(g) \circ \bifun{F}_1(f)\right) \circ \Phi_0(X) \\
    & \left(\Phi_0(Z) \circ \bifun{F}_0(g) \right) \circ \bifun{F}_0(f) \ar[ur, "\Phi_0(g)^{-1} * \bifun{F}_0(f)" swap] & &
    \bifun{F}_1(g) \circ \left(\Phi_0(Y) \circ \bifun{F}_0(f) \right) \ar[ur, "\bifun{F}_1(g) * \Phi_0(f)^{-1}" swap] &
\end{tikzcd}.\]

When $f = \Id_X$, $\Phi_0(\Id_X)$ can be pre-composed with
the 2-isomorphism $\bifun{F}_1^0(X) \colon \Id_{\bifun{F}_1(X)} \Isoto \bifun{F}_1(\Id_X)$ and unitors in
\[\begin{tikzcd}[arrows={Rightarrow, "\sim" sloped}, column sep=-0.6em]
    \Phi_0(X) \circ \Id_{\bifun{F}_0(X)} \ar[dr] &&
    \Id_{\bifun{F}_1(X)} \circ \Phi_0(X) \ar[dr, "\bifun{F}_1^0(X) * \Phi_0(X)" swap] &&
    \Phi_0(X) \circ \bifun{F}_0(\Id_X) \\
    & \Phi_0(X) \ar[ur] &&
    \bifun{F}_1(\Id_X) \circ \Phi_0(X) \ar[ur, "\Phi_0(\Id_X)" swap] &
\end{tikzcd}.\]
The full faithfulness of $\Phi_0(X)$ ensures again the existence of an unique
\[ \bifun{F}_0^0(X) \colon \Id_{\bifun{F}_0(X)} \Isoto \bifun{F}_0(\Id_X) \]
whose post-whiskering $\Phi_0(X) * \bifun{F}_0^0(X)$ is equal to the above 2-isomorphism.

If $\phi$ is a 2-morphism between two parallel morphisms $f_1, f_2 \colon X \to Y$,
it induces a 2-morphism
\[ \bifun{F}_1(\phi) * \Phi_0(X) \colon \bifun{F}_1(f_1) \circ \Phi_0(X) \Rightarrow \bifun{F}_1(f_2) \circ \Phi_0(X). \]
The universal property of $(F_0(Y), \Phi_0(Y))$ provides a 2-morphism
$\bifun{F}_0(\phi) \colon \bifun{F}_0(f_1) \Rightarrow \bifun{F}_0(f_2)$ such that
\[ \left(\Phi_0(Y) * \bifun{F}_0(\phi)\right) \circ \Phi_0(f_1) = \Phi_0(f_2) \circ \left(\bifun{F}_1(\phi) * \Phi_0(X)\right). \]
This equality exactly means that $\Phi_0$ is natural.
\end{proof}

Being an inverter in $\bicat{B}^{\bicat{A}}$ \emph{does not} automatically imply being an inverter pointwise.
It is however the case if $\bicat{B}$ admits all inverters.

\section{Triangle of biadjunctions}\label{sec:biadj_triangle}

Let be a triangle of biadjunctions
\begin{equation}
\begin{tikzcd}
    \bicat{A}
        \dar[shift left, "\bifun{F}", ""{coordinate, name=F}]
        \ar[dr, shift left=3.5, "\bifun{H}"{inner sep=0.2ex}, ""{coordinate, name=H}] & \\
    \bicat{B}
        \uar[shift left=1.5, "\bifun{L}", ""{coordinate, name=L}]
        \rar["\bifun{G}"] &
    \bicat{C}
        \ar[ul, shift right, "\bifun{N}"{inner sep=0.2ex}, ""{coordinate, name=N}]
    \ar[from=L, to=F, phantom, "\dashv"{marking, font=\scriptsize}]
    \ar[from=N, to=H, phantom, "\dashv"{marking, font=\scriptsize}]
\end{tikzcd}
\end{equation}
with a commutation (up to a pseudonatural equivalence) $\Gamma \colon \bifun{G F} \Isoto \bifun{H}$.
Our aim is to build a pseudofunctor $\bifun{M} \colon \bicat{C} \to \bicat{B}$ left biadjoint to $\bifun{G}$.
Denote by $\Phi$ the unit of $\bifun{L} \dashv \bifun{F}$, by $\Psi$ the unit of $\bifun{N} \dashv \bifun{H}$ and by $\Upsilon$ its counit.

The pseudofunctor $\bifun{M}$ will be an inverter in the bicategory $\bicat{B}^{\bicat{C}}$ of pseudofunctors from $\bicat{C}$ to $\bicat{B}$.
For the moment, the most basic pseudofunctors from $\bicat{C}$ to $\bicat{B}$ at our disposal are $\bifun{F N}$ and $\bifun{(F L)(F N)}$.
They are linked by the pseudonatural transformation $\Phi_{\bifun{F N}}$.
Let us build an other pseudonatural transformation $\Xi \colon \bifun{F N} \Rightarrow \bifun{(F L)(F N)}$.

The pseudonatural equivalence $\Gamma$ permits to turn the biadjunction $\bifun{N} \dashv \bifun{H}$
into a biadjunction $\bifun{N} \dashv \bifun{G F}$ with unit and counit respectively
\[ \Psi' \colon \ID_{\bicat{C}} \xRightarrow{\Psi} \bifun{H N} \xRightarrow{\Gamma^{-1}_{\bifun{N}}} \bifun{(G F) N}, \qquad
\Upsilon' \colon \bifun{N} (\bifun{G F}) \xRightarrow{\bifun{N} \Gamma} \bifun{N H} \xRightarrow{\Upsilon} \ID_{\bicat{A}}. \]
This biadjunction produces a first pseudonatural transformation
\[ \Theta \colon \bifun{N G} \xRightarrow{\bifun{N} (\Phi')}
\bifun{N} ((\bifun{G F}) \bifun{L}) \Isoto
(\bifun{N} (\bifun{G F})) \bifun{L} \xRightarrow{\Upsilon'_{\bifun{L}}} \bifun{L} \]
from
\[ \Phi' \colon \bifun{G}  \xRightarrow{\bifun{G} \Phi} \bifun{G} (\bifun{F L}) \Isoto (\bifun{G F}) \bifun{L} \]
where the unnamed pseudonatural equivalences are the appropriate associators.
Pre-whiskering this $\Theta$ by $\bifun{F}$ and precomposing it with an associator
yield an other pseudonatural transformation
\begin{equation}\label{eq:Theta'}
    \Theta' \colon (\bifun{F N}) \bifun{G} \Isoto \bifun{F} (\bifun{N G}) \xRightarrow{\bifun{F} \Theta} \bifun{F L}.
\end{equation}
Postcomposing $\Psi'$ by an associator provides
\[ \Psi'' \colon \ID_{\bicat{C}} \xRightarrow{\Psi'} \bifun{(G F) N} \Isoto \bifun{G} (\bifun{F N}). \]
Now, we set
\[ \Xi \colon \bifun{F N} \xRightarrow{\Xi'} ((\bifun{F N}) \bifun{G}) (\bifun{F N}) \xRightarrow{\Theta'_{\bifun{F N}}} (\bifun{F L}) (\bifun{F N}) \]
with
\[ \Xi' \colon \bifun{F N} \xRightarrow{\bifun{F N} (\Psi'')} (\bifun{F N}) (\bifun{G} (\bifun{F N})) \Isoto ((\bifun{F N}) \bifun{G}) (\bifun{F N}). \]

Let us assume that there is a modification $\mu$ filling
\begin{equation}
\begin{tikzcd}[arrows=Rightarrow, column sep=large]
    \bifun{F N}
        \rar[bend left, start anchor=north east, end anchor=north west, "\Phi_{\bifun{F N}}", ""{coordinate, name=PhiFN}]
        \rar[bend right, start anchor=south east, end anchor=south west, "\Xi" below, ""{coordinate, name=Xi}]
        \ar[from=PhiFN, to=Xi, nfold=3, shorten=0.2em, "\mu"] &
    (\bifun{F L})(\bifun{F N})
\end{tikzcd}.
\end{equation}
Let us also assume that this $\mu$ has an inverter
$(\bifun{M}, \Lambda \colon \bifun{M} \Rightarrow \bifun{F N})$
in $\bicat{B}^{\bicat{C}}$. 
What are the conditions for this $\bifun{M}$ to be left biadjoint to $\bifun{G}$?
We need to build successively: an unit; a counit; left and right triangle laws.
Additional conditions will be required at each step,
but finally we will end up with the following theorem.

\begin{theorem}\label{thm:biadj_triangle}
With the above notations,
assume the following conditions:
\begin{enumerate}
\item there is a modification $\mu \colon \Phi_{\bifun{F N}} \Rrightarrow \Xi$
admitting an inverter $(\bifun{M}, \Lambda)$ in $\bicat{B}^{\bicat{C}}$;
\item\label{it:G_pres_inv} $\bifun{G}$ preserves this inverter, i.e.
$(\bifun{G M}, \bifun{G}(\Lambda))$ is still an inverter of $\bifun{G}(\mu)$;
\item\label{it:Gmu_inv} $\bifun{G}(\mu) * \Psi''$ is invertible;
\item\label{it:Phi_desc_type} $(\ID_{\bicat{B}}, \Phi)$ is an inverter of a modification $\nu \colon \Phi_{\bifun{FL}} \Rrightarrow \bifun{F L}(\Phi)$
such that
\begin{equation}\label{eq:comp_mu_nu}
\sbox{\tmpcontent}{$(\bifun{F N}) \bifun{G}$}
\begin{tikzcd}[arrows=Rightarrow, column sep=scriptsize]
    \usebox{\tmpcontent}
        \rar[bend left, "(\Phi_{\bifun{F N}})_{\bifun{G}}", ""{coordinate, name=PhiFNG}]
        \rar[bend right, "\Xi_{\bifun{G}}"{name=XiG1}, ""{coordinate, name=XiG2}]
        \ar[from=PhiFNG, to=XiG1, nfold=3, shorten <=0.3em, "\mu_{\bifun{G}}" swap]
        \dar{\Theta'} &
    ((\bifun{F L}) (\bifun{F N})) \bifun{G} \dar{\Theta''} \\
    |[text width=\wd\tmpcontent, text depth=\dp\tmpcontent, align=center]| \bifun{F L} \rar[bend right, "\bifun{F L} (\Phi)"{name=FLPhi}] \ar[from=XiG2, to=FLPhi, nfold=3, shorten <=0.3em, "\sim" sloped, "\phi_2" swap] &
    (\bifun{F L}) (\bifun{F L})
\end{tikzcd} =
\begin{tikzcd}[arrows=Rightarrow, column sep=scriptsize]
    \usebox{\tmpcontent} \rar[bend left, "(\Phi_{\bifun{F N}})_{\bifun{G}}", ""{coordinate, name=PhiFNG}] \dar{\Theta'} & ((\bifun{F L}) (\bifun{F N})) \bifun{G} \dar{\Theta''} \\
    |[text width=\wd\tmpcontent, text depth=\dp\tmpcontent, align=center]| \bifun{F L} \rar[bend left, "\Phi_{\bifun{F L}}"{name=PhiFL1}, ""{coordinate, name=PhiFL2}] \rar[bend right, "\bifun{F L} (\Phi)"{name=FLPhi}] \ar[from= PhiFNG, to= PhiFL1, nfold=3, shorten <=0.3em, "\sim"{sloped, swap}, "\phi_1" swap] \ar[from= PhiFL2, to=FLPhi, nfold=3, shorten <=0.3em, "\nu" swap] &
(\bifun{F L}) (\bifun{F L})
\end{tikzcd};
\end{equation}
\item $\bifun{G}$ preserves the fully faithfulness of $\Phi$, i.e. $\bifun{G}(\Phi)$ remains fully faithful;
\item $\Phi_{\bifun{F N}}$ is also fully faithful.
\end{enumerate}
Then $\bifun{M}$ constitutes a left biadjoint of $\bifun{G}$.
\end{theorem}

The pseudonatural transformation $\Theta''$ and
the modifications $\phi_1$ and $\phi_2$ appearing in condition~\ref{it:Phi_desc_type}
are going to be built later in Subsection~\ref{subsec:counit}.
Condition~\ref{it:Phi_desc_type} implies that $\Phi$ is fully faithful,
autrement dit $\bifun{L}$ is \emph{locally faithful}.
Being locally faithful means that for any objects $B_1$ and $B_2$ of $\bicat{B}$, the functor
\[ \bifun{L}_{B_1, B_2} \colon \biHom_{\bicat{B}}(B_1, B_2) \to \biHom_{\bicat{B}}(\bifun{L}(B_1), \bifun{L}(B_2)) \]
is fully faithful.

Frequently, the inverter $(\bifun{M}, \Lambda)$ comes from objectwise inverters,
i.e. whatever the object $C$ of $\bicat{C}$,
the 2-morphism $\mu_C \colon \bifun{F N}(C) \Rightarrow (\bifun{F L})(\bifun{F N})(C)$ has an inverter
\[ \Lambda_C \colon M(C) \to \bifun{F N}(C) \]
in $\bicat{B}$.
According to the previous section, these data combine into a pseudofunctor $\bifun{M}$
with a pseudonatural transformation
\[ \Lambda \colon \bifun{M} \Rightarrow \bifun{F N} \]
making up an inserter of $\mu$ in $\bicat{B}^{\bicat{C}}$. 
This stronger hypothesis allows to weaken the other hypotheses of the preceding theorem, leading to the following.

\begin{theorem}\label{thm:objectwise_biadj_triangle}
With the above notations,
assume the following conditions:
\begin{enumerate}
\item there is a modification $\mu \colon \Phi_{\bifun{F N}} \Rrightarrow \Xi$ such that,
for any object $C$ of $\bicat{C}$, $\mu_C$ has an inverter
\[ (M(C), \Lambda_C \colon M(C) \to \bifun{F N}(C)) \]
in $\bicat{B}$;
\item\label{it:G_pres_all_inv} $\bifun{G}$ transforms any inverter in $\bicat{B}$ into an inverter in $\bicat{C}$;
\item $\bifun{G}(\mu) * \Psi''$ is invertible;
\item there is a modification $\nu \colon \Phi_{\bifun{FL}} \Rrightarrow \bifun{F L}(\Phi)$ satisfying
\eqref{eq:comp_mu_nu}
and such that for any object $B$, $(B, \Phi_B)$ is an inverter of $\nu_B$.
\end{enumerate}
Then the various $M(C)$, $C$ object of $\bicat{C}$, combine into a pseudofunctor $\bifun{M}$
which constitutes a left biadjoint of $\bifun{G}$.
\end{theorem}
As it will appear from the construction,
condition~\ref{it:G_pres_all_inv} can only be checked for the inverters $(M(C), \Lambda_C)$ and $(B, \Phi_B)$.

In what follows, we will start by assuming only the existence of $\mu$.
The required assumptions will be added as they arise.

\subsection{The unit}\label{subsec:unit}

First of all, $\bifun{G}$ has to preserve the weighted limits existing in $\bicat{C}$,
because otherwise it could never have a left biadjoint.
Therefore, $\bifun{G}(\Lambda) \colon \bifun{G M} \Rightarrow \bifun{G}(\bifun{F N})$ has to be an inverter of $\bifun{G}(\mu) \colon \bifun{G}(\Phi_{\bifun{F N}}) \Rrightarrow \bifun{G}(\Xi)$.
The unit of the sought biadjunction will be a pseudonatural transformation $\ID_{\bicat{C}} \Rightarrow \bifun{G M}$.

If the modification $\bifun{G}(\mu) * \Psi''$ is invertible,
the universal property of inverter ensures the existence of a pseudonatural transformation
and an invertible modification filling the diagram
\begin{equation}
\begin{tikzcd}[arrows=Rightarrow]
    \ID_{\bicat{C}} \ar[dr, "\Psi''", ""{coordinate, name=Psi}] \dar[dashed, "\Omega"] & \\
    \bifun{G M} \rar[""{coordinate, name=GLambda}, "\bifun{G}(\Lambda)" swap] & \bifun{G(F N)}
    \ar[from=Psi, to=GLambda, nfold=3, bend right, shorten=0.25em, "\sim" sloped, "\omega"{near start}]
\end{tikzcd}.
\end{equation}
The invertibility of $\bifun{G}(\mu) * \Psi''$ is thus part of the conditions of our theorem.

\subsection{The counit}\label{subsec:counit}

Having constructed the unit of the biadjunction, the second step is now to build its counit,
which is a pseudonatural transformation $\bifun{M G} \Rightarrow \ID_{\bicat{B}}$.
To do so, $\bifun{F}$ must satisfy a bicategorical equivalent of being of descent type.
One possibility would be to require that
the unit $\Phi \colon \ID_{\bicat{B}} \Rightarrow \bifun{F L}$ constitutes
a descent object of the truncated pseudomonad associated to the biadjunction $\bifun{L} \dashv \bifun{F}$.
That is what \cite{Nunes_2016} does.

We will instead assume the existence a modification $\nu$ filling the diagram
\begin{equation}
\begin{tikzcd}[arrows=Rightarrow, column sep=large]
    \bifun{F L}
        \rar[bend left, start anchor=north east, end anchor=north west, "\Phi_{\bifun{F L}}", ""{coordinate, name=PhiFL}]
        \rar[bend right, start anchor=south east, end anchor=south west, "(\bifun{F L})(\Phi)" below, ""{coordinate, name=FLPhi}]
        \ar[from=PhiFL, to=FLPhi, nfold=3, shorten=0.2em, "\nu"] &
    (\bifun{F L})(\bifun{F L})
\end{tikzcd}
\end{equation}
of which $(\ID_{\bicat{B}}, \Phi)$ constitutes an inverter.

To obtain a pseudonatural transformation $\ID_{\bicat{B}} \Rightarrow \bifun{F L}$ by universal property of inverters,
additional hypotheses have to be added.
First, we notice the existence of dashed pseudonatural transformations and invertible modifications in
\[\begin{tikzcd}[arrows=Rightarrow, column sep=large]
    \bifun{F L}
        \ar[rrr, bend left, start anchor=north east, end anchor=north west, "\Phi_{\bifun{F L}}", ""{coordinate, name=PhiFL}]
        \ar[rrr, bend right, start anchor=south east, end anchor=south west, "\bifun{F L} (\Phi)" below, ""{coordinate, name=FLPhi}] &
    (\bifun{F N}) \bifun{G}
        \lar[dashed, "\Theta'" swap]
        \rar[bend left, start anchor=north east, end anchor=north west, "(\Phi_{\bifun{F N}})_{\bifun{G}}"{name=PhiFNG}]
        \rar[bend right, start anchor=south east, end anchor=south west, "\Xi_{\bifun{G}}"{below, name=XiG}]
        \ar[from=PhiFNG, to=PhiFL, dashed, nfold=3, shorten >=0.2em, "\sim"{sloped, near start}, "\phi_1" swap]
        \ar[from=XiG, to=FLPhi, dashed, nfold=3, shorten >=0.2em, "\sim"{sloped, near start}, "\phi_2"{swap, near start}] &
    (\bifun{(F L)(F N)}) \bifun{G} \rar[dashed, "\Theta''"] &
    (\bifun{F L}) (\bifun{F L})
\end{tikzcd}.\]
The pseudonatural transformation $\Theta'$ has already been built in \eqref{eq:Theta'}.
Using $\Theta'$ and an associator of pseudofunctors composition,
one sets
\begin{equation}
    \Theta'' \colon ((\bifun{F L})(\bifun{F N})) \bifun{G} \Isoto (\bifun{F L})((\bifun{F N}) \bifun{G}) \xRightarrow{(\bifun{F L}) (\Theta')} (\bifun{F L})(\bifun{F L}).
\end{equation}
The modification $\phi_1$ is given by
\[\begin{tikzcd}[arrows=Rightarrow, column sep=tiny]
    & ((\bifun{F L})(\bifun{F N})) \bifun{G} \ar[dr, "\sim" sloped, "\Alpha_{\bifun{F L}, \bifun{F N}, \bifun{G}}"{near end}] & \\[+0.9em]
    (\bifun{F N}) \bifun{G}
        \ar[ur, "(\Phi_{\bifun{F N}})_{\bifun{G}}"]
        \ar[rr, "\Phi_{(\bifun{F N}) \bifun{G}}"{name=PhiFNG1}, ""{coordinate, name=PhiFNG2}]
        \dar{\Theta'}
        \ar[from=ur, to=PhiFNG1, nfold=3, shorten <=0.3em, "\sim"{sloped, swap}, "\Alpha_{\Phi, \bifun{F N}, \bifun{G}}^{-1}"] & &
    (\bifun{F L})((\bifun{F N}) \bifun{G}) \dar["\bifun{F L}(\Theta')"] \\
    \bifun{F L} \ar[rr, "\Phi_{\bifun{F L}}"{name=PhiFL}] & & (\bifun{F L})(\bifun{F L})
        \ar[from=PhiFNG2, to=PhiFL, nfold=3, shorten <=0.2em, "\sim" sloped, "\Phi_{\Theta'}" swap]
\end{tikzcd}.\]
The construction of $\phi_2$ is more involved and will be postponed for the moment.

Now we postulate the equality \eqref{eq:comp_mu_nu} which can also be written
\[ \phi_2 \circ (\Theta'' * \mu_{\bifun{G}}) = (\nu * \Theta') \circ \phi_1. \]
It follows that, denoting by $\alpha$ the appropriate associators in $\bicat{B}^{\bicat{B}}$,
\begin{align*}
    \nu * (\Theta' \circ \Lambda_{\bifun{G}})
    &= \alpha_{(\bifun{F L})\Phi, \Theta', \Lambda_{\bifun{G}}}
    \circ \left((\nu * \Theta') * \Lambda_{\bifun{G}} \right)
    \circ \alpha_{\Phi_{\bifun{F L}}, \Theta', \Lambda_{\bifun{G}}}^{-1} \\
    &= \alpha_{(\bifun{F L})\Phi, \Theta', \Lambda_{\bifun{G}}}
    \circ \left( \phi_2 \circ (\Theta'' * \mu_{\bifun{G}}) \circ \phi_1^{-1} \right) * \Lambda_{\bifun{G}}
    \circ \alpha_{\Phi_{\bifun{F L}}, \Theta', \Lambda_{\bifun{G}}}^{-1}
\end{align*}
and
\begin{multline*}
    \left( \phi_2 \circ (\Theta'' * \mu_{\bifun{G}}) \circ \phi_1^{-1} \right) * \Lambda_{\bifun{G}} =
    (\phi_2 * \Lambda_{\bifun{G}}) \circ \alpha_{\Theta'', \Xi_{\bifun{G}}, \Lambda_{\bifun{G}}}^{-1}
    \circ \left( \Theta'' * (\mu_{\bifun{G}} * \Lambda_{\bifun{G}}) \right) \\
    \circ \alpha_{\Theta'', (\Phi_{\bifun{F N}})_{\bifun{G}}, \Lambda_{\bifun{G}}} \circ (\phi_1^{-1} * \Lambda_{\bifun{G}}).
\end{multline*}
Furthermore, the pseudofunctoriality of the pre-whiskering by $\bifun{G}$
furnishes invertible modifications $\delta_1$ and $\delta_2$ such that
\[ \mu_{\bifun{G}} * \Lambda_{\bifun{G}} = \delta_2^{-1} \circ (\mu * \Lambda)_{\bifun{G}} \circ \delta_1. \]
Therefore, $\mu_{\bifun{G}} * \Lambda_{\bifun{G}}$ and $\nu * (\Theta' \circ \Lambda_{\bifun{G}})$
are invertible, hence the existence of $\Sigma$ and $\sigma$ filling
\[\begin{tikzcd}[arrows=Rightarrow]
    \bifun{M G} \rar{\Lambda_{\bifun{G}}} \dar[dashed]{\Sigma} &
    (\bifun{F N}) \bifun{G} \dar["\Theta'", ""{coordinate, near start, name=Theta}] \\
    \ID_{\bicat{B}} \rar["\Phi"{name=Phi}] & \bifun{F L}
    \ar[from=Theta, to=Phi, bend right, dashed, shorten <=0.4em, nfold=3, "\sim" sloped, "\sigma"]
\end{tikzcd}.\]

Let us return to the construction of $\phi_2$.
We will construct intermediary modifications and assemble them.
First, we have the following diagram
\[\sbox{\tmpcontent}{$(((\bifun{F N}) \bifun{G}) (\bifun{F N})) \bifun{G}$}
\begin{tikzcd}[arrows=Rightarrow, column sep={\the\dimexpr\wd\tmpcontent*3/2+2.4em\relax,between origins}]
    (\bifun{F N}) \bifun{G} \rar{\Xi'_{\bifun{G}}} \ar[rr, bend left=25, "\Xi_{\bifun{G}}", ""{coordinate, name=XiG}] \ar[ddr, bend right=20, "((\bifun{F N}) \bifun{G}) \Phi"{description, name=FNGPhi}] \ar[dd, "\Theta'" swap] &
    \usebox{\tmpcontent} \rar["((\Theta')_{\bifun{F N}})_{\bifun{G}}", ""{name=ThetaFNG3}] \dar["\sim" sloped] \ar[from=XiG, nfold=3, shorten <=0.3em, "\sim"{sloped, swap}, "\Alpha_{\Theta', \bifun{F N}, \bifun{G}}^{-1}"] \ar[to=FNGPhi, bend right=10, dashed, nfold=3, "\sim" sloped] &
    ((\bifun{F L})(\bifun{F N})) \bifun{G} \dar["\sim" sloped] \\
    & ((\bifun{F N}) \bifun{G})((\bifun{F N}) \bifun{G}) \rar["(\Theta')_{(\bifun{F N}) \bifun{G}}"{name=ThetaFNG}] \dar{((\bifun{F N}) \bifun{G}) \Theta'} \ar[from=ThetaFNG3, to=ThetaFNG, nfold=3, shorten <=0.3em, "\sim" sloped] & (\bifun{F L})((\bifun{F N}) \bifun{G}) \dar{(\bifun{F L}) \Theta'} \\
    |[text depth=\dp\tmpcontent]| \bifun{F L} \ar[rr, bend right=35, "(\bifun{F L}) \Phi"{name=FLPhi}] &
    ((\bifun{F N}) \bifun{G})(\bifun{F L}) \rar["(\Theta')_{\bifun{F L}}"{name=ThetaFL}] \ar[from=ThetaFNG, to=ThetaFL, nfold=3, shorten <=0.3em, "\sim"{sloped, swap}, "(\Theta')_{\Theta'}"] \ar[to=FLPhi, nfold=3, "\sim" sloped, "(\Theta')_{\Phi}^{-1}"] &
    (\bifun{F L})(\bifun{F L})
\end{tikzcd}\]
whose dashed arrow is the only one which remains to construct.

By associativity of whiskering, $\Xi'_{\bifun{G}}$ is isomorphic to
\[ (\bifun{F N}) \bifun{G} \xRightarrow{\bifun{F N}(\Psi''_{\bifun{G}})}
(\bifun{F N}) ((\bifun{G} (\bifun{F N})) \bifun{G}) \Isoto
((\bifun{F N}) (\bifun{G} (\bifun{F N}))) \bifun{G} \Isoto
(((\bifun{F N}) \bifun{G}) (\bifun{F N})) \bifun{G}. \]
A pentagonator combined with the pseudonaturality of a right associator gives
\[\sbox{\tmpcontent}{$(\bifun{F N}) (\bifun{G} ((\bifun{F N}) \bifun{G}))$}
\begin{tikzcd}[arrows=Rightarrow, column sep={\the\wd\tmpcontent,between origins}, row sep=large]
    &[-6.4em] ((\bifun{F N}) (\bifun{G} (\bifun{F N}))) \bifun{G}
        \ar[rr, "(\Alpha_{\bifun{F N}, \bifun{G}, \bifun{F N}}^{-1})_{\bifun{G}}"{swap, name=Assoc, inner sep=1.2ex}] &[-3.em]&[-3.em]
    (((\bifun{F N}) \bifun{G}) (\bifun{F N})) \bifun{G}
        \ar[dr, "\Alpha_{(\bifun{F N}) \bifun{G}, \bifun{F N}, \bifun{G}}" description] &[-6.4em] &[+1.2em] \\
    (\bifun{F N}) ((\bifun{G} (\bifun{F N})) \bifun{G})
        \ar[ur, "\Alpha_{\bifun{F N}, \bifun{G} (\bifun{F N}), \bifun{G}}^{-1}" description]
        \ar[drr, "(\bifun{F N}) (\Alpha_{\bifun{G}, \bifun{F N}, \bifun{G}})" description] &&&&
    ((\bifun{F N}) \bifun{G}) ((\bifun{F N}) \bifun{G}) \rar["((\bifun{F N}) \bifun{G}) (\Theta')"{inner sep=1.2ex}]
        \ar[dr, nfold=3, "\sim" sloped] &
    ((\bifun{F N}) \bifun{G}) (\bifun{F L}) \\
    && \usebox{\tmpcontent}
        \ar[urr, "\Alpha_{\bifun{F N}, \bifun{G}, (\bifun{F N}) \bifun{G}}^{-1}" description]
        \ar[rrr, "(\bifun{F N}) (\bifun{G} \Theta')"{swap, inner sep=1.2ex}] \ar[from=Assoc, nfold=3, "\sim" sloped] &&&
    |[xshift=-4.8em]| (\bifun{F N}) (\bifun{G} (\bifun{F L})) \ar[u, "\Alpha_{\bifun{F N}, \bifun{G}, \bifun{F L}}^{-1}" swap]
\end{tikzcd}.\]
Composing these two invertible modifications and gathering the pseudonatural transformations sharing $\bifun{F N}$,
we get an invertible modification
\[ \phi'_2 \colon ((\bifun{F N}) \bifun{G}) \Theta' \circ \Alpha_{(\bifun{F N}) \bifun{G}, \bifun{F N}, \bifun{G}} \circ \Xi'_{\bifun{G}} \Iisoto
\Alpha_{\bifun{F N}, \bifun{G}, \bifun{F L}}^{-1} \circ (\bifun{F N}) (\bifun{G} \Theta' \circ \Alpha \circ \Psi''_{\bifun{G}}). \]
We are thus reduced to fill
\[\begin{tikzcd}[arrows=Rightarrow, column sep=tiny, row sep=large]
    & (((\bifun{F N}) \bifun{G}) (\bifun{F N})) \bifun{G}
        \ar[rr, "\sim" sloped, "\Alpha_{(\bifun{F N}) \bifun{G}, \bifun{F N}, \bifun{G}}"{name=Assoc, swap}] &[-1.2em]&[-1.2em]
    ((\bifun{F N}) \bifun{G}) ((\bifun{F N}) \bifun{G})
        \ar[dr, "((\bifun{F N}) \bifun{G})(\Theta')"] \\
    (\bifun{F N}) \bifun{G}
        \ar[ur, end anchor={[xshift=1em]south west}, "\Xi'_{\bifun{G}}"]
        \ar[rr, bend left=12, "(\bifun{F N}) (\bifun{G} \Theta' \circ \Alpha \circ \Psi''_{\bifun{G}})", ""{coordinate, name=FNGTheta}]
        \ar[rr, bend right=12, "(\bifun{F N}) (\bifun{G} \Phi)"{description, name=FNGPhi1}]
        \ar[from= FNGTheta, to=FNGPhi1, dashed, nfold=3, shorten <=0.2em, "\sim"{sloped, swap}, "\bifun{F N}(\phi''_2)"]
        \ar[rrrr, bend right=20, end anchor=south west, "((\bifun{F N}) \bifun{G}) \Phi"{name=FNGPhi2}] &&
    (\bifun{F N}) (\bifun{G} (\bifun{F L}))
        \ar[rr, "\sim" sloped, "\Alpha_{\bifun{F N}, \bifun{G}, \bifun{F L}}^{-1}" swap]
        \ar[from=Assoc, nfold=3, "\sim"{sloped, swap}, "\phi'_2"]
        \ar[to=FNGPhi2, nfold=3, "\sim" sloped, "(\Alpha^{-1})_{\bifun{F N}, \bifun{G}, \Phi}^{-1}"] &&
    ((\bifun{F N}) \bifun{G}) (\bifun{F L})
\end{tikzcd}\]
with an appropriate $\phi''_2$.

Developing $\Psi''$ in term of $\Psi'$ and $\Theta'$ in term of $\Theta$ and then applying a pentagon lead to
\[\begin{tikzcd}[arrows=Rightarrow, column sep=small, row sep=large]
    &&[-4.2em] (\bifun{G} (\bifun{F N})) \bifun{G}
        \ar[rr, "\Alpha_{\bifun{G}, \bifun{F N}, \bifun{G}}"{swap, name=Assoc}] &[-3.em]&[-3.em]
    \bifun{G} ((\bifun{F N}) \bifun{G})
        \ar[dr, "\bifun{G} (\Alpha_{\bifun{F}, \bifun{N}, \bifun{G}})" description]
        \ar[drr, bend left, "\bifun{G}(\Theta')", ""{coordinate, name=GTheta}] &[-4.2em] \\
    \bifun{G}
        \ar[urr, bend left, "\Psi''_{\bifun{G}}", ""{coordinate, name=PsiG}]
        \rar[swap]{\Psi'_{\bifun{G}}} &
    ((\bifun{G F}) \bifun{N}) \bifun{G}
        \ar[ur, "(\Alpha_{\bifun{G}, \bifun{F}, \bifun{N}})_{\bifun{G}}" description]
        \ar[drr, "\Alpha_{\bifun{G F}, \bifun{N}, \bifun{G}}" description]
        \ar[from=PsiG, shorten <=0.3em, nfold=3, "\sim" sloped] &&&&
    \bifun{G} (\bifun{F} (\bifun{N G})) \rar{\bifun{G} (\bifun{F} \Theta)}
        \ar[from=GTheta, shorten <=0.3em, nfold=3, "\sim" sloped]
        \dar[nfold=3, "\sim" sloped] &
    \bifun{G} (\bifun{F L}) \\
    &&& (\bifun{G F}) (\bifun{N G})
        \ar[urr, "\Alpha_{\bifun{G}, \bifun{F}, \bifun{N G}}" description]
        \ar[rr, "(\bifun{G F}) \Theta"] \ar[from=Assoc, nfold=3, "\sim" sloped] &&
    (\bifun{G F}) \bifun{L} \ar[ur, "\Alpha_{\bifun{G}, \bifun{F}, \bifun{L}}" swap]
\end{tikzcd}.\]
Next, developing $\Theta$ and modifying the order of whiskering of $\Phi'$,
the following invertible modification is obtained
\[\begin{tikzcd}[arrows=Rightarrow, column sep=scriptsize, row sep=large]
    \bifun{G} \rar["\Psi'_{\bifun{G}}", ""{coordinate, xshift=-.5ex, name=PsiG}] \dar[swap]{\Phi'} &
    ((\bifun{G F}) \bifun{N}) \bifun{G} \rar["\Alpha_{\bifun{G F}, \bifun{N}, \bifun{G}}", ""{coordinate, name=Assoc1}] \dar["((\bifun{G F}) \bifun{N}) \Phi'" description] &
    (\bifun{G F}) (\bifun{N G}) \rar["(\bifun{G F}) \Theta", ""{coordinate, xshift=-1ex, name=GFTheta}] \dar["(\bifun{G F}) (\bifun{N} \Phi')" description] &
    (\bifun{G F}) \bifun{L} \\
    (\bifun{G F}) \bifun{L} \rar["\Psi'_{(\bifun{G F}) \bifun{L}}"{name=PsiGFL}] \ar[from=PsiG, to=PsiGFL, shorten <=0.3em, nfold=3, "\sim" sloped, "\Psi'_{\Phi'}" swap] &
    ((\bifun{G F}) \bifun{N}) ((\bifun{G F}) \bifun{L}) \rar["\Alpha_{\bifun{G F}, \bifun{N}, (\bifun{G F}) \bifun{L}}"{inner sep=1.2ex, name=Assoc2}] \ar[from=Assoc1, to=Assoc2, shorten <=0.3em, nfold=3, "\sim" sloped] &
    (\bifun{G F}) (\bifun{N} ((\bifun{G F}) \bifun{L})) \rar["(\bifun{G F})(\Alpha_{\bifun{N}, \bifun{G F}, \bifun{L}}^{-1})"{inner sep=1.2ex, name=GFAssoc}] \ar[from=GFTheta, to=GFAssoc, shorten <=0.3em, nfold=3, "\sim"{sloped, swap}] &
    (\bifun{G F}) ((\bifun{N} (\bifun{G F})) \bifun{L}) \uar[swap]{(\bifun{G F})(\Upsilon'_{\bifun{L}})}
\end{tikzcd}.\]
Since by definition of $\Phi'$ there is also an invertible modification
$\Alpha_{\bifun{G}, \bifun{F}, \bifun{L}} \circ \Phi' \Iisoto \bifun{G} \Phi$,
exhibiting an invertible modification
\[ (\bifun{G F})(\Upsilon'_{\bifun{L}}) \circ (\bifun{G F})(\Alpha_{\bifun{N}, \bifun{G F}, \bifun{L}}^{-1})
\circ \Alpha_{\bifun{G F}, \bifun{N}, (\bifun{G F}) \bifun{L}} \circ \Psi'_{(\bifun{G F}) \bifun{L}}
\Iisoto \Id_{(\bifun{G F}) \bifun{L}} \]
will conclude the construction of $\phi''_2$.

Using again a pentagonator and the pseudonaturality of the associators,
one obtains an invertible modification
\[\sbox{\tmpcontent}{$(\bifun{G F}) (\bifun{N} ((\bifun{G F}) \bifun{L}))$}
\begin{tikzcd}[arrows=Rightarrow, column sep={\the\wd\tmpcontent,between origins}, row sep=large]
    &&[-6.4em]&[-3em] \usebox{\tmpcontent}
        \ar[drr, "(\bifun{G F}) (\Alpha_{\bifun{N}, \bifun{G F}, \bifun{L}}^{-1})"] &[-3em]&[-6.4em]& \\
    (\bifun{G F}) \bifun{L}
        \rar{\Psi'_{(\bifun{G F}) \bifun{L}}}
        \ar[drr, bend right, "\left(\Psi'_{\bifun{G F}}\right)_{\bifun{L}}" swap, ""{coordinate, name=PsiGFL}] &
    ((\bifun{G F}) \bifun{N}) ((\bifun{G F}) \bifun{L})
        \ar[urr, "\Alpha_{\bifun{G F}, \bifun{N}, (\bifun{G F}) \bifun{L}}"]
        \ar[dr, "\Alpha_{(\bifun{G F}) \bifun{N}, \bifun{G F}, \bifun{L}}^{-1}" description]
        \ar[to=PsiGFL, shorten >=0.3em, nfold=3, "\sim" sloped] &&&&
    (\bifun{G F}) ((\bifun{N} (\bifun{G F})) \bifun{L})
        \rar["(\bifun{G F})(\Upsilon'_{\bifun{L}})"{inner sep=1.2ex}] &
    (\bifun{G F}) \bifun{L} \\
    && (((\bifun{G F}) \bifun{N}) (\bifun{G F})) \bifun{L}
        \ar[rr, "(\Alpha_{\bifun{G F}, \bifun{N}, \bifun{G F}})_{\bifun{L}}"{inner sep=1.2ex, name=Assoc}]
        \ar[from=uur, to=Assoc, nfold=3, "\sim" sloped] &&
    ((\bifun{G F}) (\bifun{N} (\bifun{G F}))) \bifun{L}
        \ar[ur, "\Alpha_{\bifun{G F}, \bifun{N} (\bifun{G F}), \bifun{L}}" description]
        \ar[urr, bend right, "\left((\bifun{G F})\Upsilon'\right)_{\bifun{L}}" swap, ""{coordinate, name=GFYL}]
        \ar[from=ur, to=GFYL, shorten >=0.3em, nfold=3, "\sim" sloped]
\end{tikzcd}.\]
All the transformations composing the lowest pseudonatural transformation of the above diagram
share the same right whisker by $\bifun{L}$, which can therefore be factored.
Moreover, since $\Psi'$ and $\Upsilon'$ are respectively the unit and counit of a biadjunction $\bifun{N} \dashv \bifun{G F}$,
the associated right triangle law furnishes the last invertible modification completing $\phi''_2$.

\subsection{The triangle laws}\label{subsec:triangle_laws}

To complete the construction of the biadjunction, it remains to verify the left and right triangle laws.
In the context of biadjunctions, these laws are translated into two invertible modifications
\begin{gather}
\begin{tikzcd}[arrows=Rightarrow, ampersand replacement=\&]
    \&[-1.2em] \bifun{M} (\bifun{G M}) \rar["\sim" sloped, ""{coordinate, name=Assoc}] \&
    (\bifun{M G}) \bifun{M} \ar[dr, "\Sigma_{\bifun{M}}"] \&[-1.2em] \\
    \bifun{M} \ar[ur, "\bifun{M}(\Omega)"] \ar[rrr, "\Id_{\bifun{M}}"{name=IdM}] \& \& \& \bifun{M}
    \ar[from=Assoc, to=IdM, nfold=3, shorten <=0.3em, dashed, "\sim" sloped]
\end{tikzcd}
\intertext{and}
\begin{tikzcd}[arrows=Rightarrow, ampersand replacement=\&]
    \&[-1.2em] (\bifun{G M}) \bifun{G} \rar["\sim" sloped, ""{coordinate, name=Assoc}] \&
    \bifun{G} (\bifun{M G}) \ar[dr, "\bifun{G}(\Sigma)"] \&[-1.2em] \\
    \bifun{G} \ar[ur, "\Omega_{\bifun{G}}"] \ar[rrr, "\Id_{\bifun{G}}"{name=IdG}] \& \& \& \bifun{G}
    \ar[from=Assoc, to=IdG, nfold=3, shorten <=0.3em, dashed, "\sim" sloped]
\end{tikzcd}.
\end{gather}

\paragraph{Left triangle law}
The pseudonatural transformation $\Phi_{\bifun{F N}} \circ \Lambda$ is fully faithful as
the composition of two such transformations.
It reduces the problem to building an invertible morphism in
$\biHom_{\bicat{C}^{\bicat{B}}}(\bifun{M}, (\bifun{F L}) (\bifun{F N}))$
from
\[
    (\Phi_{\bifun{F N}} \circ \Lambda) \circ (\Sigma_{\bifun{M}} \circ \Alpha_{\bifun{M}, \bifun{G}, \bifun{M}}^{-1} \circ \bifun{M}(\Omega))
\]
to $\Phi_{\bifun{F N}} \circ \Lambda$.
Thanks to
\[ \Phi_{\Lambda} \colon (\bifun{F L})(\Lambda) \circ \Phi_{\bifun{M}} \Iisoto \Phi_{\bifun{F N}} \circ \Lambda, \]
it is the same as building an invertible modification
\begin{equation}
    ((\bifun{F L})(\Lambda) \circ \Phi_{\bifun{M}}) \circ (\Sigma_{\bifun{M}} \circ \Alpha_{\bifun{M}, \bifun{G}, \bifun{M}}^{-1} \circ \bifun{M}(\Omega))
    \Iisoto \Phi_{\bifun{F N}} \circ \Lambda.
\end{equation}

By construction of $\Sigma$, there is an invertible modification
\[ ((\bifun{F L})(\Lambda) \circ \Phi_{\bifun{M}}) \circ (\Sigma_{\bifun{M}} \circ \Alpha_{\bifun{M}, \bifun{G}, \bifun{M}}^{-1} \circ \bifun{M}(\Omega)) \Iisoto
(\bifun{F L})(\Lambda) \circ (\Theta' \circ \Lambda_{\bifun{G}})_{\bifun{M}} \circ \Alpha_{\bifun{M}, \bifun{G}, \bifun{M}}^{-1} \circ \bifun{M}(\Omega). \]
The invertible modification $(\Theta' \circ \Lambda_{\bifun{G}})_{\Lambda}$ produces an exchange
\[ (\bifun{F L})(\Lambda) \circ (\Theta' \circ \Lambda_{\bifun{G}})_{\bifun{M}}
\circ \Alpha_{\bifun{M}, \bifun{G}, \bifun{M}}^{-1} \circ \bifun{M}(\Omega) \Iisoto
(\Theta' \circ \Lambda_{\bifun{G}})_{\bifun{F N}} \circ (\bifun{M G})(\Lambda)
\circ \Alpha_{\bifun{M}, \bifun{G}, \bifun{M}}^{-1} \circ \bifun{M}(\Omega) \]
and $(\Alpha^{-1})_{\bifun{M}, \bifun{G}, \Lambda}$ followed by a factorization of the left whiskering by $\bifun{M}$ yields
\[ (\Theta' \circ \Lambda_{\bifun{G}})_{\bifun{F N}} \circ (\bifun{M G})(\Lambda)
\circ \Alpha_{\bifun{M}, \bifun{G}, \bifun{M}}^{-1} \circ \bifun{M}(\Omega) \Iisoto
(\Theta' \circ \Lambda_{\bifun{G}})_{\bifun{F N}} \circ \Alpha_{\bifun{M}, \bifun{G}, \bifun{F N}}^{-1} \circ \bifun{M}(\bifun{G}(\Lambda) \circ \Omega). \]
By construction of $\Omega$, there is invertible modification from the latter pseudonatural transformation to
\[ (\Theta' \circ \Lambda_{\bifun{G}})_{\bifun{F N}} \circ \Alpha_{\bifun{M}, \bifun{G}, \bifun{F N}}^{-1} \circ \bifun{M}(\Psi''). \]
Developing $(\Theta' \circ \Lambda_{\bifun{G}})_{\bifun{F N}}$, rearranging right whiskerings
and exchanging left and right whiskerings produce a compound invertible modification
\[\begin{tikzcd}[arrows=Rightarrow]
    \bifun{M} \rar["\bifun{M}(\Psi')", ""{coordinate, xshift=-2pt, name=MPsi}] \dar[swap]{\Lambda} &
    \bifun{M} (\bifun{G} (\bifun{F N})) \rar["\sim", ""{coordinate, name=Assoc1}]
        \dar{\Lambda_{\bifun{G} (\bifun{F N})}} &
    (\bifun{M G}) (\bifun{F N})
        \dar{(\Lambda_{\bifun{G}})_{\bifun{F N}}}
        \ar[dr, bend left=20, "(\Theta' \circ \Lambda_{\bifun{G}})_{\bifun{F N}}", ""{coordinate, name=ThetaLambda}] \\
    \bifun{F N} \rar["(\bifun{F N})(\Psi')"{name=FNPsi}]
        \ar[from=MPsi, to=FNPsi, nfold=3, shorten <=0.3em, "\sim" sloped, "(\Lambda_{\Psi'})^{-1}"] &
    (\bifun{F N}) (\bifun{G} (\bifun{F N})) \rar["\sim"{name=Assoc2}]
        \ar[from=Assoc1, to=Assoc2, nfold=3, shorten <=0.3em, "\sim" sloped] &
    ((\bifun{F N}) \bifun{G}) (\bifun{F N}) \rar["(\Theta')_{\bifun{F N}}"{name=Theta}]
        \ar[from=ThetaLambda, to=Theta, bend right=20, nfold=3, shorten <=0.3em, "\sim" sloped] &
    (\bifun{F L}) (\bifun{F N})
\end{tikzcd}.\]
Moreover, it holds
\[ (\mu * \Lambda)^{-1} \colon (\Theta')_{\bifun{F N}} \circ \Alpha_{\bifun{F N}, \bifun{G}, \bifun{F N}}^{-1} \circ (\bifun{F N})(\Psi') \circ \Lambda = \Xi \circ \Lambda \Iisoto \Phi_{\bifun{F N}} \circ \Lambda, \]
which completes our construction.

\paragraph{Right triangle law}
The construction of the right triangle law is simpler,
because it only exploits the full faithfulness of $\bifun{G}(\Phi)$.
It reduces the problem to find an invertible modification from
\begin{equation}\label{eq:biadj_triangle_r_lhs}
    \bifun{G}(\Phi) \circ (\bifun{G}(\Sigma) \circ \Alpha_{\bifun{G}, \bifun{M}, \bifun{G}} \circ \Omega_{\bifun{G}})
\end{equation}
to $\bifun{G}(\Phi)$.
We have
\[\begin{tikzcd}[arrows=Rightarrow, column sep=normal, row sep=large]
    &&& \bifun{G} \ar[dr, "\bifun{G}(\Phi)"] & \\
    & (\bifun{G M}) \bifun{G}
        \rar["\Alpha_{\bifun{G}, \bifun{M}, \bifun{G}}"]
        \ar[dr, "(\bifun{G} \Lambda)_{\bifun{G}}"{swap, inner sep=-0.2ex}] &
    \bifun{G}(\bifun{M G})
        \ar[dr, "\bifun{G}(\Lambda_{\bifun{G}})" swap]
        \ar[ur, "\bifun{G}(\Sigma)"]
        \ar[rr, "\bifun{G}(\Phi \circ \Sigma)"{description, xshift=3.5pt, name=GPhiSigma}]
        \ar[from=ur, to=GPhiSigma, nfold=3, "\sim" sloped]
        \dar[nfold=3, "\sim" sloped, "\Alpha_{\bifun{G}, \Lambda, \bifun{G}}"{swap, near start}] &&
    \bifun{G} (\bifun{F L}) \\
    \bifun{G} \ar[ur, "\Omega_{\bifun{G}}"]
        \ar[rr, "(\bifun{G}(\Lambda) \circ \Omega)_{\bifun{G}}"{description, name=GLOG}]
        \ar[rr, bend right, end anchor=south west, "\Psi''_{\bifun{G}}" swap, ""{coordinate, name=PsiG}]
        \ar[from=ur, to=GLOG, nfold=3, "\sim" sloped]
        \ar[from=GLOG, to=PsiG, nfold=3, shorten >=0.3em, "\sim"{sloped, near start}, "\omega_{\bifun{G}}"{pos=0.4, inner sep=1ex}] &&
    (\bifun{G} (\bifun{F N})) \bifun{G} \rar["\sim", "\Alpha_{\bifun{G}, \bifun{F N}, \bifun{G}}"{swap, inner sep=1.ex}] &
    \bifun{G} ((\bifun{F N}) \bifun{G})
        \ar[ur, "\bifun{G}(\Theta')" swap]
        \ar[from=GPhiSigma, nfold=3, "\sim" sloped, "\bifun{G}(\sigma)^{-1}"{swap, near start}] &
\end{tikzcd}.\]
So far has been built an invertible modification from \eqref{eq:biadj_triangle_r_lhs} to
\[ \bifun{G}(\Theta') \circ \Alpha_{\bifun{G}, \bifun{F N}, \bifun{G}} \circ \Psi''_{\bifun{G}}. \]
Now, $\phi''_2$ can be directly applied to conclude.

\section{Case of a fully faithful pseudofunctor}\label{sec:ff_case}

\subsection{A first observation}

When $\Theta$ is an equivalence (within the bicategory $\bicat{A}^{\bicat{B}}$),
so is $\Theta'$ and its quasi-inverse $\Theta'^{-1}$ allows to build
a modification $\nu \colon \Phi_{\bifun{F L}} \Rrightarrow \bifun{F L} (\Phi)$
directly from the data of $\mu \colon \Phi_{\bifun{F N}} \Rrightarrow \Xi$ as
\begin{equation}\label{eq:nu_from_mu}
    \Phi_{\bifun{F L}} \Iisoto
    (\Phi_{\bifun{F L}} \circ \Theta') \circ \Theta'^{-1}
        \xRightarrow{\left( \phi_2 \circ (\Theta'' * \mu_{\bifun{G}}) \circ \phi_1^{-1} \right) * \Theta'^{-1}}
    (\bifun{F L} (\Phi) \circ \Theta') \circ \Theta'^{-1} \Iisoto
    \bifun{F L} (\Phi).
\end{equation}
Equation \eqref{eq:comp_mu_nu} is clearly satisfied by this $\nu$.
In that setting, conditions of Theorem~\ref{thm:biadj_triangle} can be expressed only in term of $\mu$.

According to the following proposition, this notably occurs when $\bifun{G}$ is fully faithful.
We will see later that if the conditions of Theorem~\ref{thm:biadj_triangle} are fulfilled, this reciprocally implies that $\bifun{G}$ is fully faithful.

\begin{proposition}
If $\bifun{G}$ is fully faithful, then $\Theta$ is a pseudonatural equivalence.
\end{proposition}
\begin{proof}
Throughout this proof, $\Chi$ will be the counit of $\bifun{L} \dashv \bifun{F}$.
One needs first to exhibit an hypothetical quasi-inverse of $\Theta$.
There is a pseudonatural transformation
\begin{equation}\label{eq:GT}
\bifun{G} \xRightarrow{\Psi''_{\bifun{G}}} (\bifun{G} (\bifun{F N})) \bifun{G}
\Isoto \bifun{G} ((\bifun{F N}) \bifun{G})
\end{equation}
hence a pseudonatural transformation
\[ \Tau \colon \ID_{\bicat{B}} \Rightarrow \bifun{F} (\bifun{N G}). \]
Set
\[ \Theta^{-1} \colon \bifun{L} \xRightarrow{\bifun{L} \Tau} \bifun{L} (\bifun{F} (\bifun{N G})) \Isoto
(\bifun{L F}) (\bifun{N G}) \xRightarrow{\Chi_{\bifun{N G}}} \bifun{N G}. \]

Now let's build an invertible modification from $\Theta^{-1} \circ \Theta$ to $\Id_{\bifun{N G}}$.
An exchange of left and right whiskerings and two rearrangements provide an invertible modification
\[\begin{tikzcd}[arrows=Rightarrow]
    \bifun{L} \rar["\bifun{L} \Tau", ""{coordinate, name=LT, xshift=-.4ex}] &
    \bifun{L} (\bifun{F} (\bifun{N G})) \rar["\sim", ""{coordinate, name=Assoc1}] &
    (\bifun{L F}) (\bifun{N G}) \rar{\Chi_{\bifun{N G}}} &
    \bifun{N G} \\
    \bifun{N G} \uar{\Theta} \rar["\bifun{N G}(\Tau)"{name=NGT}]
\ar[from=LT, to=NGT, nfold=3, shorten <=0.3em, "\sim" sloped, "\Theta_{\Tau}"] &
    (\bifun{N G}) (\bifun{F} (\bifun{N G})) \uar[swap]{\Theta_{\bifun{F} (\bifun{N G})}} \rar["\sim"{name=Assoc2}] \ar[from=Assoc1, to=Assoc2, nfold=3, shorten <=0.3em, "\sim" sloped] &
    ((\bifun{N G}) \bifun{F}) (\bifun{N G}) \uar{(\Theta_{\bifun{F}})_{\bifun{N G}}} \ar[ur, bend right=20, "(\Chi \circ \Theta_{\bifun{F}})_{\bifun{N G}}" swap, ""{coordinate, name=ChiThetaFNG}]
\ar[from=u, to=ChiThetaFNG, nfold=3, shorten >=0.3em, "\sim" sloped]
\end{tikzcd}.\]
The expansion of $\Theta$ in term of $\Phi$ and $\Upsilon'$ followed by exchanges and a pentagonator produces an invertible modification
from $\Chi \circ \Theta_{\bifun{F}}$ to
\[ (\bifun{N G}) \bifun{F} \xRightarrow{(\bifun{N G})(\Alpha_{\bifun{F}, \bifun{L}, \bifun{F}} \circ \Phi_{\bifun{F}})}
(\bifun{N G}) (\bifun{F} (\bifun{L F})) \xRightarrow{(\bifun{N G}) (\bifun{F} \Chi)}
(\bifun{N G}) \bifun{F} \xRightarrow{\Upsilon' \circ \Alpha_{\bifun{N}, \bifun{G}, \bifun{F}}} \ID_{\bicat{A}}. \]
The right triangle law of the biadjunction $\bifun{L} \dashv \bifun{F}$ maps isomorphically
$\bifun{F} \Chi \circ \Alpha_{\bifun{F}, \bifun{L}, \bifun{F}} \circ \Phi_{\bifun{F}}$
to $\Id_{(\bifun{N G}) \bifun{F}}$.
Combining the preceding invertible modifications,
an invertible modification from $\Theta^{-1} \circ \Theta$ to
\begin{equation}\label{eq:UpsilonNG_NGT}
\bifun{N G} \xRightarrow{\bifun{N G}(\Tau)} (\bifun{N G}) (\bifun{F} (\bifun{N G})) \Isoto
((\bifun{N G}) \bifun{F}) (\bifun{N G}) \xRightarrow{(\Upsilon' \circ \Alpha_{\bifun{N}, \bifun{G}, \bifun{F}})_{\bifun{N G}}} \bifun{N G}
\end{equation}
has been built so far.
By construction, $\bifun{N G}(\Tau)$ is isomorphic to
\[ \bifun{N G} \xRightarrow{\bifun{N} (\Psi'_{\bifun{G}})}
\bifun{N} (((\bifun{G F}) \bifun{N}) \bifun{G}) \Isoto
\bifun{N} (\bifun{G} (\bifun{F} (\bifun{N G}))) \Isoto
(\bifun{N G}) (\bifun{F} (\bifun{N G})). \]
Thus, a pentagonator and exchanges and rearrangements of whiskering assembles in
an invertible modification from \eqref{eq:UpsilonNG_NGT} to
\[ \bifun{N G} \xRightarrow{(\bifun{N} \Psi')_{\bifun{G}}}
(\bifun{N} ((\bifun{G F}) \bifun{N})) \bifun{G} \Isoto
((\bifun{N} (\bifun{G F})) \bifun{N}) \bifun{G} \xRightarrow{\left(\Upsilon'_{\bifun{N}}\right)_{\bifun{G}}}
\bifun{N G}. \]
Since $\Psi'$ and $\Upsilon'$ are respectively the unit and counit of a biadjunction $\bifun{N} \dashv \bifun{G F}$,
its left triangle law maps isomorphically this last pseudonatural transformation to $\Id_{\bifun{N G}}$,
completing our construction.

To finish the proof, it only remains to build an invertible modification from $\Theta \circ \Theta^{-1}$ to $\Id_{\bifun{G N}}$.
First, we have
\[\begin{tikzcd}[arrows=Rightarrow]
    \bifun{L} \rar["\bifun{L} \Tau", ""{coordinate, name=LT, xshift=-.4ex}]
        \ar[dr, bend right, "\bifun{L} (\bifun{F} \Theta \circ \Tau)" swap, ""{coordinate, name=LFThetaT}] &
    \bifun{L} (\bifun{F} (\bifun{N G}))
        \rar["\sim", ""{coordinate, name=Assoc1}] \dar{\bifun{L} (\bifun{F} \Theta)}
        \ar[to= LFThetaT, nfold=3, "\sim" sloped, shorten >=0.3em] &
    (\bifun{L F}) (\bifun{N G})
        \rar["\Chi_{\bifun{N G}}", ""{coordinate, name=XNG, xshift=-.4ex}]
        \dar{(\bifun{L F}) \Theta} &
    \bifun{N G} \dar{\Theta} \\
    & \bifun{L} (\bifun{F L}) \rar["\sim"{name=Assoc2}]
        \ar[from=Assoc1, to=Assoc2, nfold=3, shorten <=0.3em, "\sim" sloped] &
    (\bifun{L F}) \bifun{L} \rar["\Chi_{\bifun{L}}"{name=XL}]
        \ar[from=XNG, to=XL, nfold=3, shorten <=0.3em, "\sim" sloped, "\Chi_{\Theta}"] &
    \bifun{L}
\end{tikzcd}.\]
If we are able to build an invertible modification $\bifun{F} \Theta \circ \Tau \Iisoto \Phi$,
then we are done thanks to the left triangle law of the biadjunction $\bifun{L} \dashv \bifun{F}$.

By full faithfulness of $\bifun{G}$, that is equivalent to build an invertible modification $\bifun{G}(\bifun{F} \Theta \circ \Tau) \Iisoto \bifun{G}(\Phi)$.
Replacing $\bifun{G}(\Tau)$ by \eqref{eq:GT} and rearranging the left whiskers, one gets
\[ \bifun{G}(\bifun{F} \Theta \circ \Tau) \Iisoto \Alpha_{\bifun{G}, \bifun{F}, \bifun{L}} \circ (\bifun{G F}) \Theta \circ \Alpha_{\bifun{G F}, \bifun{N}, \bifun{G}} \circ \Psi'_{\bifun{G}}. \]
Then developing $\Theta$ and performing some exchanges, one obtains
\[\begin{tikzcd}[arrows=Rightarrow]
    \bifun{G} \rar["\Psi'_{\bifun{G}}", ""{coordinate, xshift=-.4ex, name=PsiG}] \dar[swap]{\Phi'} &
    ((\bifun{G F}) \bifun{N}) \bifun{G} \rar["\Alpha_{\bifun{G F}, \bifun{N}, \bifun{G}}", ""{coordinate, name=Assoc1}]
\dar[swap]{((\bifun{G F}) \bifun{N}) \Phi'} &
    (\bifun{G F}) (\bifun{N G}) \dar[swap]{(\bifun{G F}) (\bifun{N} \Phi')}
        \ar[ddd, bend left=67, "(\bifun{G F}) \Theta", ""{coordinate, name=GFTheta}] \\
    (\bifun{G F}) \bifun{L} \rar["\Psi'_{(\bifun{G F}) \bifun{L}}"{name=PsiGFL1}, ""{coordinate, name=PsiGFL2}]
\ar[from=PsiG, to=PsiGFL1, nfold=3, shorten <=0.3em, "\sim" sloped, "\Psi'_{\Phi'}"]
\dar[swap]{\left(\Psi'_{(\bifun{G F})}\right)_{\bifun{L}}} &
    ((\bifun{G F}) \bifun{N}) ((\bifun{G F}) \bifun{L}) \rar["\Alpha_{\bifun{G F}, \bifun{N}, (\bifun{G F}) \bifun{L}}"{name=Assoc2, inner sep=1ex}] \ar[from=Assoc1, to=Assoc2, nfold=3, shorten <=0.3em, shorten >=-.4ex, "\sim" sloped]
\ar[dl, "\Alpha_{(\bifun{G F}) \bifun{N}, \bifun{G F}, \bifun{L}}^{-1}"{near end}, ""{coordinate, name=Assoc3}]
\ar[from=PsiGFL2, to=Assoc3, nfold=3, shorten=0.3em, bend right, "\sim" sloped]
\ar[dd, nfold=3, "\sim" sloped] &
    (\bifun{G F}) (\bifun{N} ((\bifun{G F}) \bifun{L}))
\dar["(\bifun{G F})(\Alpha_{\bifun{N}, \bifun{G F}, \bifun{L}}^{-1})" swap, ""{coordinate, name=GFAssoc}]
\ar[from=GFTheta, to= GFAssoc, nfold=3, shorten=0.3em, "\sim" sloped] \\
    (((\bifun{G F}) \bifun{N}) (\bifun{G F})) \bifun{L} \ar[dr, "(\Alpha_{\bifun{G F}, \bifun{N}, \bifun{G F}})_{\bifun{L}}" swap] &&
    (\bifun{G F}) ((\bifun{N} (\bifun{G F})) \bifun{L}) \dar["(\bifun{G F})(\Upsilon'_{\bifun{L}})" swap] \\
& ((\bifun{G F}) (\bifun{N} (\bifun{G F}))) \bifun{L} \ar[ur, "\Alpha_{\bifun{G F},(\bifun{N} (\bifun{G F}), \bifun{L}}"{near end}, ""{coordinate, name=Assoc4}] \rar["((\bifun{G F})\Upsilon')_{\bifun{L}}" swap, ""{coordinate, name=GFUpsilon}]
\ar[from=Assoc4, to=GFUpsilon, nfold=3, shorten=0.3em, bend left, "\sim" sloped] &
(\bifun{G F}) \bifun{L}
\end{tikzcd}\]
and the right triangle law of the biadjunction $\bifun{N} \dashv \bifun{G F}$ concludes.
\end{proof}

When $\Theta$ is an equivalence, condition~\ref{it:Phi_desc_type} can be reformulated as follows.

\begin{proposition}
Keep the same notations as before and assume that $\Theta$ is an equivalence
and that $\mu$ admits an inverter $(\bifun{M}, \Lambda)$.
Then the following assertions are equivalent:
\begin{enumerate}
\item\label{it:inverter_nu_from_mu}
    $(\ID_{\bicat{B}}, \Phi)$ is an inverter of the $\nu$ defined in \eqref{eq:nu_from_mu};
\item\label{it:inverter_muG}
    $(\ID_{\bicat{B}}, \Theta'^{-1} \circ \Phi)$ is an inverter of $\mu_{\bifun{G}}$;
\item\label{it:Sigma_equivalence}
    there is a pseudonatural equivalence $\Sigma \colon \bifun{M G} \Isoto \ID_{\bicat{B}}$ and
    an invertible modification $\sigma$ filling
\[\begin{tikzcd}[arrows=Rightarrow]
    \bifun{M G} \rar{\Lambda_{\bifun{G}}} \dar[dashed, "\Sigma" swap, "\sim" sloped] &
    (\bifun{F N}) \bifun{G} \dar["\Theta'", ""{coordinate, near start, name=Theta}] \\
    \ID_{\bicat{B}} \rar["\Phi"{name=Phi}] & \bifun{F L}
    \ar[from=Theta, to=Phi, bend right, dashed, shorten <=0.4em, nfold=3, "\sim" sloped, "\sigma"]
\end{tikzcd}.\]
\end{enumerate}
\end{proposition}
\begin{proof}
Everything is a consequence of the following more general result.
\end{proof}

\begin{lemma}
Let $\bicat{C}$ be any bicategory in which there is a diagram
\[\begin{tikzcd}[row sep=large]
    X_0 \rar{i}
    &
    X_1
        \rar[bend left, start anchor=north east, end anchor=north west, "f_0"{description, near start, name=f0}]
        \dar["h_1"{swap, name=h1}] &[+1em]
    X_2 \dar{h_2} \\
    Y_0 \rar["j" name=j]
    &
    Y_1
        \rar[bend left, start anchor=north east, end anchor=north west, "g_0"{description, near start, name=g0}]
        \ar[from=f0, to=g0, Rightarrow, shorten=0.2em, "\sim"{near start, sloped}]
        \ar[from=u, to=ur, crossing over, bend right, start anchor=south east, end anchor=south west, "f_1"{description, near end, name=f1}]
        \ar[from=f0, to=f1, Rightarrow, shorten=0.2em, "\alpha"]
        \rar[bend right, start anchor=south east, end anchor=south west, "g_1"{description, near end, name=g1}]
        \ar[from=g0, to=g1, Rightarrow, shorten=0.2em, "\beta" swap] &
    Y_2
        \ar[from=f1, to=g1, crossing over, Rightarrow, shorten=0.2em, "\sim"{swap, sloped}]
\end{tikzcd}\]
where $h_1$ and $h_2$ are equivalences and
\[ \eta_2 \circ (h_2 * \alpha) = (\beta * h_1) \circ \eta_1. \]
\begin{itemize}
\item If there exists an equivalence $h_0 \colon X_0 \isoto Y_0$ and
a 2-isomorphism $h_1 \circ i \Isoto j \circ h_0$, then
$(X_0, i)$ is an inverter of $\alpha$ if and only if $(Y_0, j)$ is an inverter of $\beta$.
\item If $(X_0, i)$ and $(Y_0, j)$ are both inverters, then
there exists an equivalence $h_0 \colon X_0 \isoto Y_0$
(unique up to a compatible 2-isomorphism) and
a 2-isomorphism $h_1 \circ i \Isoto j \circ h_0$.
\end{itemize}
\end{lemma}
\begin{proof}
Let denote by $h_1^{-1}$ (resp. $h_2^{-1}$) a quasi-inverse of $h_1$ (resp. $h_2$).
Without further hypotheses, Proposition~\ref{prop:inverter_morphism} already states that
$i * \alpha$ is invertible if and only if $j * \beta$ is so.
We will therefore suppose that it is the case until the end of this proof.

Let us assume the existence of a $h_0 \colon X_0 \isoto Y_0$ and
of a 2-isomorphism $h_1 \circ i \Isoto j \circ h_0$.
The latter induces for any object $W$ of $\bicat{C}$ an invertible transformation
between the canonical functors
\[\begin{tikzcd}
    \Hom_{\bicat{C}}(W, X_0)
        \rar["i \circ {-}", ""{coordinate, name=i}]
        \dar{h_0 \circ {-}} &
    \InvCone(W, \alpha) \dar{h_1 \circ {-}} \\
    \Hom_{\bicat{C}}(W, Y_0)
        \rar["j \circ {-}"{name=j}]
        \ar[from=i, to=j, shorten <=0.2em, Rightarrow, "\sim" sloped] &
    \InvCone(W, \alpha)
\end{tikzcd}.\]
Since $h_0$ and $h_1$ are equivalences in $\bicat{C}$,
one of the two horizontal functors is an equivalence of categories if and only if the other is so,
hence the first assertion.

Assume now that $(X_0, i)$ and $(Y_0, j)$ are both inverters.
According to Proposition~\ref{prop:inverter_morphism},
there are a $h_0 \colon X_0 \to Y_0$ and a $\eta_0 \colon h_1 \circ i \Isoto j \circ h_0$.
The same proposition, applied with $h_1^{-1}$ and $h_2^{-1}$ instead of $h_1$ and $h_2$,
provides $h'_0 \colon Y_0 \to X_0$ and $\eta'_0 \colon h_1^{-1} \circ j \Isoto i \circ h'_0$.
The diagram
\[\begin{tikzcd}
    && X_1 \dar[Rightarrow, "\sim" sloped, "\eta_0" swap] \ar[dr, "j"] \\
    Y_0 \rar{h'_0} \ar[dr, "j" swap] & X_0 \rar{i} \ar[ur, "h_0"] \dar[Rightarrow, "\sim" sloped, "\eta'_0" swap] & X_1 \rar{h_1} & Y_1 \\
    & Y_1 \ar[ur, "h_1^{-1}" swap] \ar[urr, bend right=24, "\Id_{Y_1}"{swap, near end}, ""{coordinate, near end, name=Id}]
\ar[from=ur, to=Id, Rightarrow, shorten >=0.3em, "\sim" sloped]
\end{tikzcd}\]
describes a 2-isomorphism $j \circ h_0 \circ h'_0 \Isoto j$.
By universal property of inverters, it implies the existence of a 2-isomorphism $h_0 \circ h'_0 \Isoto \Id_{Y_0}$.
A similar argument can be applied for $h'_0 \circ h_0$, hence the second assertion.
\end{proof}

\begin{corollary}
If, on top of the hypotheses of the previous proposition, conditions of Theorem~\ref{thm:biadj_triangle} are met,
then $\bifun{G}$ is fully faithful.
\end{corollary}
\begin{proof}
It results from the fact that the counit $\Sigma$ of the biadjunction $\bifun{M} \dashv \bifun{G}$ is an equivalence.
\end{proof}

\subsection{Restatement of Theorem~\ref{thm:objectwise_biadj_triangle}}

We are now going to concentrate on the simplification of
the statement of Theorem~\ref{thm:objectwise_biadj_triangle} when $\bifun{G}$ is fully faithful.
First, one can notice that the data of a modification $\mu \colon \Phi_{\bifun{FN}} \Rrightarrow \Xi$
is equivalent to the data of a modification $\bifun{G} (\Phi_{\bifun{FN}}) \Rrightarrow \bifun{G} \Xi$.
Then, $\phi''_2$ combined with some exchanges furnishes
\[ \bifun{G} (\Phi_{\bifun{FN}}) \Iisoto
\bifun{G} \left((\Theta')_{\bifun{FN}} \circ \Alpha_{\bifun{FN}, \bifun{G}, \bifun{FN}}^{-1}\right)
\circ \Alpha_{\bifun{G}, \bifun{FN}, \bifun{G} (\bifun{FN})} \circ \Psi''_{\bifun{G} (\bifun{FN})}. \]
Thus, simplifying by $(\Theta')_{\bifun{FN}} \circ \Alpha_{\bifun{FN}, \bifun{G}, \bifun{FN}}^{-1}$,
the data of $\mu$ is equivalent
to the data of a modification $\Alpha_{\bifun{G}, \bifun{FN}, \bifun{G} (\bifun{FN})} \circ \Psi''_{\bifun{G} (\bifun{FN})} \Rrightarrow \bifun{G} (\bifun{FN} \Psi'')$,
which is in turn equivalent to the data of a modification
$\mu' \colon \Psi_{\bifun{H N}} \Rrightarrow (\bifun{H N})\Psi$
thanks to $\Gamma$.

In addition, if $\bifun{G}$ has a left biadjoint, then a 2-morphism of $\bicat{B}$ admits an inverter in $\bicat{B}$
if and only if its image by $\bifun{G}$ has an inverter in $\bicat{C}$,
which will automatically belongs to the essential image of $\bifun{G}$.
Therefore, if Theorem~\ref{thm:objectwise_biadj_triangle} holds,
the following assertions on an object $C$ of $\bicat{C}$ become equivalent:
\begin{enumerate}
\item $C$ belongs to the essential image of $\bifun{G}$;
\item $\Omega_C \colon C \to \bifun{G M}(C)$ is an equivalence;
\item $(C, \Psi_C)$ is an inverter of $\mu'_C$.
\end{enumerate}
This leads us to the following reformulation of our theorems.

\begin{theorem}\label{thm:biadj_triangle_ff}
Keep the same notations as before.
Assume that $\bifun{G}$ fully faithful and:
\begin{enumerate}
\item there is a modification
\[\begin{tikzcd}[arrows=Rightarrow, column sep=large]
    \bifun{H N}
        \rar[bend left, start anchor=north east, end anchor=north west, "\Psi_{\bifun{H N}}", ""{coordinate, name=PsiHN}]
        \rar[bend right, start anchor=south east, end anchor=south west, "(\bifun{H N})\Psi" below, ""{coordinate, name=HNPsi}]
        \ar[from=PsiHN, to=HNPsi, nfold=3, shorten=0.2em, "\mu'"] &
    (\bifun{H N})(\bifun{H N})
\end{tikzcd}\]
such that, for any object $C$ of $\bicat{C}$, $\mu'_C$ has an inverter $(M'(C), \Lambda'_C)$ in $\bicat{C}$;
\item $\mu' * \Psi$ is invertible;
\item\label{it:ess_img_G} an object $C$ of $\bicat{C}$ is in the essential image of $\bifun{G}$ if and only if
$(C, \Psi_C)$ is an inverter of $\mu'_C$;
\item $\bifun{H N}$ preserves the pseudomonicity of $\Lambda'_C$ for any $C$.
\end{enumerate}
Then, for any object $C$ of $\bicat{C}$, $M'(C)$ is in the essential image of $\bifun{G}$.
Choosing for any $C$ an object $M(C)$ of $\bicat{B}$ so that $\bifun{G} M(C)$ is equivalent to $M'(C)$,
one gets a pseudofunctor $\bifun{M} \colon \bicat{C} \to \bicat{B}$ constituting a left biadjoint of $\bifun{G}$.
\end{theorem}
\begin{proof}
To apply Theorem~\ref{thm:objectwise_biadj_triangle},
it suffices to show that $M'(C)$ is indeed in the essential image of $\bifun{G}$ for any object $C$ of $\bicat{C}$.
According to condition~\ref{it:ess_img_G}, $M'(C)$ is in the essential image of $\bifun{G}$
if $(M'(C), \Psi_{M'(C)})$ is an inverter of $\mu'_{M'(C)}$.
However, there is the following diagram in $\bicat{C}$
\[\begin{tikzcd}[sep=huge]
    M'(C) \rar["\Lambda'_{C}", ""{coordinate, name=i, xshift=.8ex}] \dar{\Psi_{M'(C)}}
    &
    \bifun{H N}(C)
        \rar[bend left, start anchor=north east, end anchor=north west, "\Psi_{\bifun{H N}(C)}"{description, near start, name=f0}]
        \dar["\Psi_{\bifun{H N}(C)}"{swap, name=h1}] &[1.2em]
    (\bifun{H N})^2(C) \dar{\Psi_{(\bifun{H N})^2 C}} \\
    \bifun{H N}(M'(C)) \rar["\bifun{H N}\left(\Lambda'_{C}\right)" name=j]
        \ar[Rightarrow, from=i, to=j, shorten <=0.2em, "\sim"{sloped, swap}, "\Psi_{\Lambda'_{C}}^{-1}" swap]
        \dar[bend right=45, "\Psi_{\bifun{H N}(M'(C))}"{description, near start, name=k0}]
    &
    (\bifun{H N})^2(C)
        \rar[bend left, start anchor=north east, end anchor=north west, "\bifun{H N}\left(\Psi_{\bifun{H N}(C)}\right)"{description, near start, name=g0}]
        \ar[from=f0, to=g0, end anchor={[xshift=-1ex]north}, Rightarrow, shorten=0.2em, "\sim"{near start, sloped}]
        \ar[from=u, to=ur, crossing over, bend right, start anchor=south east, end anchor=south west, "\bifun{H N}(\Psi_C)"{description, near end, name=f1}]
        \ar[from=f0, to=f1, Rightarrow, shorten=0.2em, "\mu'_C"]
        \rar[bend right, start anchor=south east, end anchor=south west, "(\bifun{H N})^2 (\Psi_C)"{description, near end, name=g1}]
        \ar[from=g0, to=g1, Rightarrow, shorten=0.2em, "\bifun{H N} (\mu'_C)" swap]
        \dar[bend right=45, "\Psi_{(\bifun{H N})^2 C}"{description, near start, name=l0}]
        \dar[bend left=45, "\bifun{H N}\left(\Psi_{\bifun{H N}(C)}\right)"{description, near end, name=l1}]
        \ar[from=l0, to=l1, Rightarrow, "\mu'_{\bifun{H N}(C)}"{near end}]
        \ar[from=k0, to=l0, Rightarrow, "\sim" sloped, "\Psi_{(\bifun{H N}) \Lambda'_{C}}"{swap, near end}]
        \ar[from=l, to=dl, bend left=45, crossing over, "\bifun{H N}\left(\Psi_{M'(C)}\right)"{description, near end, name=k1}]
        \ar[from=k0, to=k1, Rightarrow, "\mu'_{M'(C)}"{near end}]
        \ar[from=k1, to=l1, Rightarrow, crossing over, "\sim" sloped] &
    (\bifun{H N})^3(C)
        \ar[from=f1, to=g1, crossing over, Rightarrow, shorten=0.2em, "\sim"{swap, sloped}] \\[1.8em]
    (\bifun{H N})^2(M'(C)) \rar[swap]{(\bifun{H N})^2 \Lambda'_{C}} &
    (\bifun{H N})^3(M'(C))
\end{tikzcd}\]
whose top row and rightmost column are inverters.
Furthermore, $\Psi_{(\bifun{H N})^2 C}$ is conservative as an inverter of $\mu'_{(\bifun{H N})^2 C}$
and $\bifun{H N}(\Lambda'_C)$ pseudomonic by hypothesis.
The following lemma thus concludes.
\end{proof}

\begin{lemma}
Let $\bicat{C}$ be any bicategory in which there is a diagram
\[\begin{tikzcd}[sep=large]
    X_0 \rar["f_0", ""{coordinate, name=i}] \dar{h_0}
    &
    X_1
        \rar[bend left, start anchor=north east, end anchor=north west, "f_1"{description, near start, name=f1}]
        \dar["i_0"{swap, name=h1}] &[1.2em]
    X_2 \dar{j} \\
    Y_0 \rar["g_0" name=j]
        \ar[Rightarrow, from=i, to=j, shorten <=0.2em, "\sim" sloped, "\eta_0" swap]
        \dar[bend right=45, "h_1"{description, near start, name=k0}]
    &
    Y_1
        \rar[bend left, start anchor=north east, end anchor=north west, "g_1"{description, near start, name=g1}]
        \ar[from=f1, to=g1, Rightarrow, shorten=0.2em, "\sim"{near start, sloped}, "\eta_1" {swap, near end}]
        \ar[from=u, to=ur, crossing over, bend right, start anchor=south east, end anchor=south west, "f_2"{description, near end, name=f2}]
        \ar[from=f1, to=f2, Rightarrow, shorten=0.2em, "\alpha"]
        \rar[bend right, start anchor=south east, end anchor=south west, "g_2"{description, near end, name=g2}]
        \ar[from=g1, to=g2, Rightarrow, shorten=0.2em, "\beta" swap]
        \dar[bend right=45, "i_1"{description, near start, name=i1}]
        \dar[bend left=45, "i_2"{description, near end, name=i2}]
        \ar[from=i1, to=i2, Rightarrow, "\delta"]
        \ar[from=k0, to=i1, Rightarrow, "\sim"{near end}, "\eta_3"{swap, near end}]
        \ar[from=l, to=dl, bend left=45, crossing over, "h_2"{description, near end, name=k1}]
        \ar[from=k0, to=k1, Rightarrow, "\gamma" swap]
        \ar[from=k1, to=i2, Rightarrow, crossing over, "\sim"{near start}, "\eta_4"{swap, near start}] &
    Y_2
        \ar[from=f2, to=g2, crossing over, Rightarrow, shorten=0.2em, "\sim"{swap, sloped}, "\eta_2" near start] \\[1.8em]
    Z_0 \rar[swap]{k} &
    Z_1
\end{tikzcd}\]
with
\[ \eta_2 \circ (j * \alpha) = (\beta * i_0) \circ \eta_1, \qquad
\eta_4 \circ (k * \gamma) = (\delta * i_0) \circ \eta_3, \]
$j$ conservative and $g_0$ pseudomonic.
If $\beta * g_0$ and $\gamma * h_0$ are invertible, $(X_0, f_0)$ is an inverter of $\alpha$ and $(X_1, i_0)$ of $\delta$,
then $(X_0, h_0)$ is an inverter of $\gamma$.
\end{lemma}
\begin{proof}
Let $W$ be an object of $\bicat{C}$.
We must show that the canonical functor
\begin{equation}\label{eq:fonct_invcone_stacked}
    \biHom_{\bicat{C}}(W, X_0) \to \InvCone(W; \gamma)
\end{equation}
has a quasi-inverse.
According to Proposition~\ref{prop:inverter_morphism}, the post-composition functor
\[ \biHom_{\bicat{C}}(W, Y_0) \xrightarrow{g_0 \circ {-}} \biHom_{\bicat{C}}(W, Y_1) \]
restricts to a functor
\[ \InvCone(W; \gamma) \to \InvCone(W; \delta) \]
which can be composed with the functor
\[ U \colon \InvCone(W; \delta) \isoto \biHom_{\bicat{C}}(W, X_1) \]
provided by the universal property of $(X_1, i_0)$,
hence a functor
\[ \InvCone(W; \gamma) \xrightarrow{U(g_0 \circ {-})} \biHom_{\bicat{C}}(W, X_1). \]

In fact, the latter takes its values in $\InvCone(W; \alpha)$,
because for any $h \colon W \to Y_0$ such that $\gamma * h$ is invertible,
\begin{align*}
    j * (\alpha * U(g_0 \circ h))
    &= \alpha_{j, f_2, U(g_0 \circ h)} \circ \left((j * \alpha) * U(g_0 \circ h)\right) \circ \alpha_{j, f_1, U(g_0 \circ h)}^{-1} \\
    &=\begin{multlined}[t]
    \alpha_{j, f_2, U(g_0 \circ h)} \circ \left(\left(\eta_2^{-1} \circ (\beta * i_0) \circ \eta_1 \right) * U(g_0 \circ h)\right) \\
    \circ \alpha_{j, f_1, U(g_0 \circ h)}^{-1}.
    \end{multlined}
\end{align*}
The universal property of $(X_1, i_0)$ provides a natural isomorphism
$\epsilon \colon \Id_{\InvCone(W; \delta)} \Isoto i_0 \circ U$, thus
\[ \beta * (i_0 \circ U(g_0 \circ h)) = (g_2 * \epsilon_{g_0 \circ h}) \circ
   (\beta * (g_0 \circ h)) \circ (g_1 * \epsilon_{g_0 \circ h}^{-1}) .\]
Since $\beta * g_0$ is invertible by hypothesis,
so are $\beta * (i_0 \circ U(g_0 \circ h))$, $j * (\alpha * U(g_0 \circ h))$
and $\alpha * U(g_0 \circ h)$ because $j$ is conservative.

Let us denote by $U'$ the obtained functor
\[ \InvCone(W; \gamma) \to \InvCone(W; \alpha) \]
and by $V$ the functor
\[ \InvCone(W; \alpha) \isoto \biHom_{\bicat{C}}(W, X_0) \]
given by the universal property of $(X_0, f_0)$.

A natural isomorphism $\Id_{\biHom_{\bicat{C}}(W, X_0)} \Isoto V U'(h_0 \circ {-})$
is built at any object $l$ of $\biHom_{\bicat{C}}(W, X_0)$
by taking the inverse image of
\[ i_0 \circ (f_0 \circ l) \Isoto (i_0 \circ f_0) \circ l \Isoto[\eta_0 * l]
(g_0 \circ h_0) \circ l \Isoto g_0 \circ (h_0 \circ l) \]
by the sequence of equivalences
\[ \Hom(l, V U (g_0 \circ (h_0 \circ l))) \isoto \Hom(f_0 \circ l, U (g_0 \circ (h_0 \circ l)))
\isoto \Hom(i_0 \circ (f_0 \circ l), g_0 \circ (h_0 \circ l)). \]

Constructing a natural isomorphism $\Id_{\InvCone(W; \gamma)} \Isoto h_0 \circ V U'({-})$
will complete our proof that $V U'$ is a quasi-inverse of \eqref{eq:fonct_invcone_stacked}.
As $g_0$ is pseudomonic, it is enough to have a natural isomorphism $g_0 \circ \Id_{\InvCone(W; \gamma)} \Isoto g_0 \circ (h_0 \circ V U'({-}))$.
But for any object $h$ of $\InvCone(W; \gamma)$, there is
\[ h \Isoto[\epsilon_h] i_0 \circ U'(h) \Isoto[i_0 * \theta_{U'(h)}] i_0 \circ (f_0 \circ V U'(h))
\Isoto[\alpha_{i_0, f_0, V U'(h)} \circ \eta_0 * V U'(h) \circ \alpha^{-1}_{g_0, h_0, V U'(h)}]
g_0 \circ (h_0 \circ V U'(h)) \]
with $\theta$ the natural isomorphism
$\Id_{\InvCone(W; \alpha)} \Isoto f_0 \circ V$
associated to the universal property of $(X_0, f_0)$.
This 2-isomorphism depends naturally of $h$, hence the end of the proof.
\end{proof}

\subsection{A biadjunction for Kleisli bicategory}

Let $\bicat{C}$ be a bicategory.
Similar to the 1-dimensional case, to any pseudomonad $\bifun{T}$ on $\bicat{C}$
is associated a \emph{Kleisli bicategory} $\Kl(\bifun{T})$ \cite{Cheng_2003}
together with a canonical biadjunction
\[\begin{tikzcd}
    \bicat{C}
        \rar[shift left=2, "\bifun{J}", ""{coordinate, name=J}] &
    \Kl(\bifun{T})
        \lar[shift left=1, "\tilde{\bifun{T}}", ""{coordinate, name=T}]
        \ar[from=T, to=J, phantom, "\vdash" marking]
\end{tikzcd}.\]
The objects of $\Kl(\bifun{T})$ are the same as those of $\bicat{C}$
and the category of morphisms between two objects $X$ and $Y$
is $\biHom_{\bicat{C}}(X, \bifun{T} Y)$.
The identity of an object is given by the unit $\Psi_X \colon X \to \bifun{T}(X)$ of the pseudomonad evaluated at $X$.
The pseudofunctor $\bifun{J}$ is the identity on objets and
sends any 1-morphism $f \colon X \to Y$ of $\bicat{C}$ to $\Psi_Y \circ f$,
whereas $\tilde{\bifun{T}}$ sends any $X$ to $\bifun{T}(X)$
and any 1-morphism $g$ of $\Kl(\bifun{T})$ from $X$ to $Y$ to $\Pi_Y \circ \bifun{T}(g)$.
We refer the reader to \cite{Cheng_2003} for more details.

A pseudomonad $(\bifun{T}, \Pi, \Psi, \beta, \eta, \pi)$ is called
a \emph{Kock-Zöberlein-monad} (abbreviated to KZ-monad)
if there exists an adjunction $\bifun{T}(\Psi) \dashv \Pi$ with unit $\eta$ \cite{Marmolejo_1997, Zoberlein_1976}.
Alternatively, a KZ-monad can be defined from the data of a modification
$\mu \colon \Psi_{\bifun{T}} \Rightarrow \bifun{T}(\Psi)$ satisfying three coherence conditions.
This is the definition of \cite{Kock_1995}.

Let $(\bifun{T}, \Pi, \Psi, \beta, \eta, \pi, \mu)$ be a KZ-monad on $\bicat{C}$
and assume also that $\bicat{C}$ has all inverters.
Taking $\bicat{A} = \Kl(\bifun{T})$, $\bifun{H} = \tilde{\bifun{T}}$ and $\bifun{N} = \bifun{J}$,
we see that the two first conditions of Theorem~\ref{thm:biadj_triangle_ff} are met.
Set $\bicat{B}$ the full sub-bicategory of $\bicat{C}$ constituted by the objects $C$
such that $(C, \Psi_C)$ is an inverter of $\mu_C$.

One can wonder if $\tilde{\bifun{T}}$ factors through $\bicat{B}$.
Indeed, the multiplication allows to build for any object $X$ of $\bicat{C}$ two additional 1-morphisms
\[ \Pi_X \colon \bifun{T}^2(X) \to \bifun{T}(X),\quad
\bifun{T}(\Pi_X) \colon \bifun{T}\left(\bifun{T}^2(X)\right) \to \bifun{T}^2(X), \]
with 2-isomorphisms
\begin{gather*}
    \beta_X \colon \Pi_X \circ \Psi_{\bifun{T}(X)} \Isoto \Id_{\bifun{T}(X)}, \\
    \bifun{T}(\Pi_X) \circ \bifun{T}\left(\Psi_{\bifun{T}(X)}\right) \Isoto[\bifun{T}_{\Pi_X, \Psi_{\bifun{T}(X)}}]
    \bifun{T}\left(\Pi_X \circ \Psi_{\bifun{T}(X)}\right) \Isoto[\bifun{T}(\eta_{X})]
    \bifun{T}\left(\Id_{\bifun{T}(X)}\right) \Isoto \Id_{\bifun{T}^2(X)}, \\
    \Psi_{\Pi_X} \colon \bifun{T}(\Pi_X) \circ \Psi_{\bifun{T}^2(X)} \Isoto \Psi_{\bifun{T}(X)} \circ \Pi_X.
\end{gather*}

Inspired by the notion of \emph{split equalizer} \cite[149]{MacLane_1978}, one can define a notion of \emph{split inverter}.
\begin{definition}
Let $\phi$ be a 2-morphism between two parallel 1-morphisms $f_1, f_2 \colon X_1 \to X_2$ of a bicategory $\bicat{C}$.
A \emph{split inverter} of $\phi$ is an inverter cone $(X_0, X_0 \xrightarrow{f_0} X_1)$
together with 1-morphisms $r_0 \colon X_1 \to X_0$, $r_1 \colon X_2 \to X_1$ and 2-isomorphisms
\[\begin{tikzcd}
    X_0 \rar{f_0} \ar[dr, "\Id_{X_0}" swap, ""{coordinate, name=id}] & X_1 \dar["r_0", ""{coordinate, near start, name=r0}] \\
    & X_0 \ar[from=r0, to=id, shorten=0.2em, Rightarrow, bend right, "\sim" sloped]
\end{tikzcd}, \quad
\begin{tikzcd}
    X_1 \rar{f_2} \ar[dr, "\Id_{X_1}" swap, ""{coordinate, name=id}] & X_2 \dar["r_1", ""{coordinate, near start, name=r1}] \\
    & X_1 \ar[from=r1, to=id, shorten=0.2em, Rightarrow, bend right, "\sim" sloped]
\end{tikzcd}, \quad
\begin{tikzcd}
    X_1 \rar{f_1} \dar{r_0} & X_2 \dar["r_1", ""{coordinate, near start, name=r1}] \\
    X_0 \rar["f_0"{name=f0}] & X_1 \ar[from=r1, to=f0, shorten <=0.2em, Rightarrow, bend right, "\sim" sloped]
\end{tikzcd}.\]
\end{definition}
\begin{proposition}
Any split inverter is an inverter.
\end{proposition}
\begin{proof}
The notation of the definition above are kept.
Let $X$ be any object of $\bicat{C}$.
The post-composition by $r_0$ gives a functor
\[ \InvCone(X; \phi) \xrightarrow{r_0 \circ {-}} \biHom_{\bicat{C}}(X, X_0). \]
We are going to show that it is quasi-inverse to the functor
\[ \biHom_{\bicat{C}}(X, X_0) \xrightarrow{f_0 \circ {-}} \InvCone(X; \phi). \]

On the one hand, the 2-isomorphism $r_0 \circ f_0 \Isoto \Id_{X_0}$ provides a natural isomorphism
\[\begin{tikzcd}
    \biHom_{\bicat{C}}(X, X_0) \rar{f_0 \circ {-}}
        \ar[dr, bend right=8, "\Id_{\biHom(X, X_0)}" swap, ""{coordinate, name=id}] &
    \InvCone(X; \phi) \dar["r_0 \circ {-}", ""{coordinate, near start, name=r0}] \\
    & \biHom_{\bicat{C}}(X, X_0) \ar[from=r0, to=id, shorten=0.2em, Rightarrow, bend right=15, "\sim" sloped]
\end{tikzcd}.\]

On the other hand, the 2-morphism obtained from the diagram
\[\begin{tikzcd}
    & X_0 \ar[dr, "f_0"] \dar[Rightarrow, "\sim" sloped] \\
    X_1 \ar[ur, "r_0"] \rar[bend left, "f_1"{description, name=f1}] \rar[bend right, "f_2"{description, name=f2}] \ar[rr, bend right=45, "\Id_{X_1}" swap, ""{coordinate, name=id}] \ar[from=f1, to=f2, Rightarrow, "\phi"] &
    X_2 \rar{r_1} \ar[to=id, Rightarrow, shorten >=0.2em, "\sim"{sloped, near start}] &
    X_1
\end{tikzcd}\]
yields a natural transformation
\[\begin{tikzcd}
    \InvCone(X; \phi) \rar{r_0 \circ {-}}
        \ar[dr, bend right=8, "\Id_{\InvCone(X; \phi)}" swap, ""{coordinate, name=id}] &
    \biHom_{\bicat{C}}(X, X_0) \dar["f_0 \circ {-}", ""{coordinate, near start, name=r0}] \\
    & \InvCone(X; \phi) \ar[from=r0, to=id, shorten=0.2em, Rightarrow, bend right=15]
\end{tikzcd}.\]
which is an isomorphism because for any $f \colon X \to X_1$ object of $\InvCone(X; \phi)$,
$\phi * f$ is invertible.
\end{proof}

\begin{corollary}
Let $X$ be an object of $\bicat{C}$ and $\bifun{T}$ a KZ-monad.
The septuple
$(\bifun{T}(X), \Psi_{\bifun{T}(X)}, \Pi_X, \bifun{T}(\Pi_X), \beta_X, \bifun{T}_{\bifun{T}(X)}^{-1} \circ \bifun{T}(\epsilon_X) \circ \bifun{T}_{\Pi_X, \Psi_{\bifun{T}(X)}}, \Psi_{\Pi_X})$
is a split inverter of $\mu_{\bifun{T}(X)}$.
\end{corollary}

The pseudofunctor $\tilde{\bifun{T}}$ thus factor through the canonical inclusion
$\bifun{I} \colon \bicat{B} \hookrightarrow \bicat{C}$, hence a (strict) triangle
\[\begin{tikzcd}
    \Kl(\bifun{T}) \dar[swap]{\tilde{\bifun{T}}'} \ar[dr, "\tilde{\bifun{T}}"{inner sep=0.2ex}] & \\
    \bicat{B} \rar[hook, "\bifun{I}"] & \bicat{C}
\end{tikzcd}.\]
It is also clear that the pseudofunctor $\bifun{J I}$ is left adjoint to $\tilde{\bifun{T}}'$.
To apply Theorem~\ref{thm:biadj_triangle_ff}, it only remains to require that $\bifun{T}$ preserves pseudomonic 1-morphisms.
To summarise,

\begin{proposition}
Let $\bicat{C}$ be a bicategory having all inverters and let
$(\bifun{T}, \Pi, \Psi, \beta, \eta, \pi, \mu)$ be a KZ-monad on it which preserves pseudomonic 1-morphisms.
The inclusion functor of the full sub-bicategory of $\bicat{C}$
made of the objects $C$ such that $(C, \Psi_C)$ is an inverter of $\mu_C$ has a left biadjoint.
\end{proposition}

\section{Dual statements}\label{sec:dual}

For completeness, let us write the dual versions of Theorem~\ref{thm:biadj_triangle}
obtained by reversing the orientation of morphisms in the bicategories.
In fact, both 1-morphisms and 2-morphisms can be reversed,
resulting in two distinct kinds of dual bicategories to treat.

\paragraph{Reversing 1-morphisms}
Let $\bicat{A}$ be a bicategory.
The bicategory obtained by this operation is called the \emph{opposite bicategory}
and is denoted by $\bicat{A}\op$.
Any inverter in $\bicat{A}$ becomes a co-inverter in $\bicat{A}\op$.

Any pseudofunctor $\bifun{F} \colon \bicat{A} \to \bicat{B}$ induces
a pseudofunctor $\bifun{F}\op \colon \bicat{A}\op \to \bicat{B}\op$ going in the same direction.
However, a pseudonatural transformation $\Phi$ between two pseudofunctors
$\bifun{F}_1, \bifun{F}_2 \colon \bicat{A} \to \bicat{B}$
induces a pseudonatural transformation
$\Phi\op \colon \bifun{F}_2\op \Rightarrow \bifun{F}_1\op$ going in the reverse direction.
For their part, modifications keep the same orientation.

This operation transforms thus left biadjoints into right biadjoints.
From that, a theorem is deduced to lift \emph{right biadjoints} instead of left ones.
\begin{theorem}
Let be a triangle of biadjunctions
\begin{equation}
\begin{tikzcd}
    \bicat{A}
        \dar[shift right=1.5, "\bifun{F}" swap, ""{coordinate, name=F}]
        \ar[dr, shift left, "\bifun{H}"{inner sep=0.2ex, swap}, ""{coordinate, name=H}] & \\
    \bicat{B}
        \uar[shift right, "\bifun{R}" swap, ""{coordinate, name=R}]
        \rar["\bifun{G}"] &
    \bicat{C}
        \ar[ul, shift right=3.5, "\bifun{T}"{inner sep=0.2ex, swap}, ""{coordinate, name=T}]
    \ar[from=F, to=R, phantom, "\dashv"{marking, font=\scriptsize}]
    \ar[from=H, to=T, phantom, "\dashv"{marking, font=\scriptsize}]
\end{tikzcd}
\end{equation}
with a commutation (up to a pseudonatural equivalence) $\Gamma \colon \bifun{G F} \Isoto \bifun{H}$.
Denote by $\Chi$ the counit of $\bifun{F} \dashv \bifun{R}$ and by $\Upsilon$ its counit.
Set
\[ \Delta \colon \bifun{L} \Rightarrow \bifun{N G}
\quad (\text{resp. } \Kappa \colon (\bifun{F R}) (\bifun{F T}) \Rightarrow \bifun{F R}) \]
the pseudonatural transformation whose dual corresponds to
the $\Theta$ (resp. $\Xi$) of Section~\ref{sec:biadj_triangle}.
Assume the following conditions:
\begin{enumerate}
\item there is a modification $\xi \colon \Chi_{\bifun{F T}} \Rrightarrow \Kappa$
admitting a co-inverter $(\bifun{S}, \Pi)$ in $\bicat{B}^{\bicat{C}}$;
\item $\bifun{G}$ preserves this co-inverter, i.e.
$(\bifun{G S}, \bifun{G}(\Pi))$ is still an inverter of $\bifun{G}(\xi)$;
\item $\Upsilon'' * \bifun{G}(\xi)$ is invertible;
\item $(\ID_{\bicat{B}}, \Chi)$ is a co-inverter of a modification $\pi \colon \Chi_{\bifun{F R}} \Rrightarrow \bifun{F R}(\Chi)$
such that
\[ \phi_2 \circ (\xi_{\bifun{G}} * \Delta'') = (\Delta' * \pi) \circ \phi_1; \]
\item $\bifun{G}$ preserves the co-fully faithfulness of $\Chi$, i.e. $\bifun{G}(\Chi)$ remains co-fully faithful;
\item $\Chi_{\bifun{F R}}$ is also co-fully faithful.
\end{enumerate}
Then $\bifun{S}$ constitutes a right biadjoint of $\bifun{G}$.
\end{theorem}

\paragraph{Reversing 2-morphisms}
Let $\bicat{A}$ be a bicategory.
The bicategory obtained by this operation is called the \emph{co-bicategory}
and is denoted by $\bicat{A}\co$.
Any inverter in $\bicat{A}$ remains an inverter in $\bicat{A}\co$.
Applying $\co$ keeps the same orientation for pseudofunctors and pseudonatural transformations,
but reverses that of modifications.
So in that case the dual statement of Theorem~\ref{thm:biadj_triangle} is the same,
except that the orientation of $\mu$ and $\nu$ is reversed.

\paragraph{Acknowledgments}
I would like to thank my thesis supervisors Olivia Caramello and Marc Aiguier
for their advice in writing this article and their careful proofreading.
I would also like to thank Axel Osmond for providing me with a number of useful references.

\printbibliography
\end{document}